\documentclass[reqno]{amsart}
\usepackage{float}
\usepackage{dsfont}
\usepackage{amsmath}
\usepackage{graphicx} 
\usepackage{latexsym}
\usepackage{amsfonts}
\usepackage{amssymb}
\usepackage{color}
\usepackage{xcolor}

\usepackage{hyperref}

\theoremstyle{plain}
\newtheorem{theorem}{Theorem}
\newtheorem{corollary}[theorem]{Corollary}
\newtheorem{lemma}[theorem]{Lemma}
\newtheorem{proposition}[theorem]{Proposition}
\theoremstyle{definition}

\newtheorem{remark}[theorem]{Remark}
\newtheorem*{remark*}{Remark}

\newcommand{\N}{\mathbb N}

\newcommand{\pr}{\mathbf P}
\newcommand{\e}{\mathbf E}
\newcommand{\E}{\mathbf E}

\newcommand{\rf}{\rfloor}
\newcommand{\lf}{\lfloor}
\newcommand{\y}{\hat{y}}

\begin{document}
\title[Expansions for normal deviations]
    {Asymptotic expansions for normal deviations of random walks conditioned to stay positive}

\thanks{Alexander Tarasov and Vitali Wachtel were supported by the DFG}

\author[Denisov]{Denis Denisov}
\address{Department of Mathematics, University of Manchester, UK}
\email{denis.denisov@manchester.ac.uk}

\author[Tarasov]{Alexander Tarasov}
\address{Faculty of Mathematics, Bielefeld University, Germany}
\email{atarasov@math.uni-bielefeld.de}
 
\author[Wachtel]{Vitali Wachtel}
\address{Faculty of Mathematics, Bielefeld University, Germany} 
\email{wachtel@math.uni-bielefeld.de}
\begin{abstract} 
We consider a one-dimensional random walk $S_n$ having i.i.d. integer-valued increments with zero mean and finite variance. We continue our study of asymptotic expansions for local probabilities $\mathbf P(S_n=x,\tau_0>n)$, started in \cite{DTW23}, where we obtain expansions valid for $x=o(\frac{\sqrt{n}}{\log^{1/2} n})$. In the present paper we derive an alternative expansion, which deals with $x$ of order $\sqrt{n}$. 
\end{abstract}
    
    
    \keywords{Random walk, exit time, asymptotic expansion}
    \subjclass{Primary 60G50; Secondary 60G40, 60F17}
    \maketitle
    {\scriptsize
    }
\section{Introduction}
Let  $\{X_k\}$ be a sequence of i.i.d.
random variables with zero mean, $\e X_1=0$, and finite variance, $\sigma^2:=\e X_1^2\in(0,\infty)$. 
Furthermore, we shall always assume that the random variables $X_k$ take values in $\mathbb{Z}$.
Consider a random walk 
 $\{S_n; n\ge0\}$ defined as follows, 
 $S_0=0$ and 
$$
S_n:=X_1+X_2+\ldots+X_n,\ n\ge1.
$$
Let $\tau$ be the first time when $S_n$ becomes non-positive, that is,
$$
\tau:=\inf\{n\ge1:S_n\le 0\}.
$$
In \cite{DTW23} we have shown that if $\E|X_1|^{r+3}$ is finite for some integer $r\ge0$ then 
\begin{align} 
\label{eq:locprobtau}
    \pr(S_n=x,\tau>n)
=
    \sum_{j=1}^{\lf \frac{r+2}{2}\rf} 
        \frac{U_j(x)}{(n+1)^{j+1/2}}
+
    O\left(
        \frac{(1+x)^{r+{2}}}{n^{(r+{4})/2}}\log^{{\lf \frac{r+{1}}{2}\rf}}n
    \right),
\end{align}
where $U_j$ are polyharmonic functions for $S_n$ killed upon leaving $(0,\infty)$. A remarkable feature of this expansion is the fact that the coefficients $U_j$ are not polynomials, as is always the case in expansions in the classical central limit theorem. Every $U_j$ is only asymptotically polynomial, see Theorem 6 in \cite{DTW23}. This result also implies that 
\begin{align}
\label{eq:U-bound}
|U_j(x)|\le C_j(1+x^{2j-1}),\quad x>0.
\end{align}
Using this estimate one can compare regular summands in \eqref{eq:locprobtau} with the remainder there. One concludes then that the expansion becomes useless for 
$x\ge \frac{\sqrt{n}}{\log^{1/2}n}$. More precisely, for such values of $x$ we only have 
$$
\pr(S_n=x,\tau>n)=
O\left(
        \frac{(1+x)^{r+1}}{n^{(r+3)/2}}\log^{{\lf \frac{r}{2}\rf}}n
    \right).
$$
Thus the case of ``large'' values of $x$ needs to be studied separately.
\begin{theorem}
\label{thm:lattice}
Assume that $\e|X_1|^{r+3}$ is finite for some integer $r \ge 0$.
Assume also that $S_n$ is integer-valued with maximal span $1$. Then there exist polynomials $P_2$, $P_3$, \ldots, $P_{r+2}$ such that
\begin{align*}
    \pr(S_n=x,\tau>n) 
&=
    e^{-\frac{x^2}{2\sigma^2n}}
    \sum_{\nu = 2}^{r+2}
        \frac{P_\nu(x/\sigma\sqrt{n})}{n^{\nu/2}}\\ 
&\hspace{2cm}
+
    o\left(
        \frac{1}{n^{(r+2)/2}}
    \right)
+
    O\left( \frac{1}{\min (\sqrt{n}, |x|)^{r+3}} \right).
\end{align*}
Every polynomial $P_\nu$ has degree $3\nu-5$.
\end{theorem}
\begin{remark}\label{rem:absolute}
We believe that the methods used to prove Theorem~\ref{thm:lattice} can be used in the absolutely continuous case and that the following analogue should hold.
If $\e|X_1|^{r+3}$ is finite for some integer $r \ge 0$ and if the distribution of $X_1$ is absolutely continuous with a bounded density, then there exist polynomials $P_2$, $P_3$, \ldots, $P_{r+2}$ such that
\begin{align*}
    \frac{d}{dx}\pr(S_n\le x,\tau>n) 
&=
    e^{-\frac{x^2}{2\sigma^2n}}
    \sum_{\nu = 2}^{r+2}
        \frac{P_\nu(x/\sigma\sqrt{n})}{n^{\nu/2}}\\ 
&\hspace{2cm}
+
    o\left(
        \frac{1}{n^{(r+2)/2}}
    \right)
+
    O\left( \frac{1}{\min (\sqrt{n}, |x|)^{r+3}} \right).
\end{align*}
Every polynomial $P_\nu$ has degree $3\nu-5$.
\hfill$\diamond$
\end{remark}
\begin{remark}
    Assume that $\E e^{\lambda|X_1|}$ is finite for some $\lambda > 0$. Then, by \cite[Theorem 6]{DTW23} there are polynomials $Q_j$ of degree $2j-1$ such that for some $\varepsilon > 0$
\begin{align*}
    U_j(x) 
=
    Q_j(x) 
+
    O(e^{-\varepsilon x}),
\end{align*}
where $U_j(x)$ are the functions from decomposition \eqref{eq:locprobtau}. Since in this case $U_j$ is almost polynomial, it is natural to compare their polynomial parts $Q_j$ with polynomials $P_\nu$ from
Theorem~\ref{thm:lattice}. This comparison allows one to express coefficients of $Q_j$ in terms of the coefficients of the polynomials $P_\nu$ for 
$\nu \le 2j+1$, for details see Section~\ref{sec:7} below.
\hfill$\diamond$
\end{remark}
These expansions can be seen as a refinement of the local limit theorem proved by Caravenna~\cite{Caravenna05}. His result for lattice walks reads as follows:
$$
\pr(S_n=x,\tau>n)
\sim c_0\frac{x}{\sigma^2n^{3/2}}e^{-\frac{x^2}{2\sigma^2n}}.
$$
Comparing this with the claim in Theorem~\ref{thm:lattice}, we conclude that 
$$
P_2(t)=c_0\frac{t}{\sigma}.
$$
In the last section we give a direct proof of this equality and describe the constant $c_0$ in terms of the harmonic function $U_1(x)$, see Corollary~\ref{cor:r=2} below.

Compared to \eqref{eq:locprobtau}, expansions in Theorem~\ref{thm:lattice} have a more standard form as they contain polynomials of $x/\sqrt{n}$. This can be explained by the fact that if $x$ is of order $\sqrt{n}$ then the behaviour is more universal than closer to the boundary. A more rigorous explanation is given by Theorem~6 in \cite{DTW23}, where we have shown that polyharmonic functions $U_j$ behave asymptotically as polynomials.

These expansions are quite similar to the expansions derived by Nagaev~\cite{Nagaev70} for the running maximum of the walk $S_n$. His approach is based on a careful analysis of the Fourier transforms. We apply a completely different method, which does not use Fourier transforms. The starting point in our analysis is the following obvious equality 
\begin{align}
\label{eq:reflection}
\pr(S_n=x,\tau>n)=\pr(S_n=x)-\pr(S_n=x,\tau< n).
\end{align}
Since the first term on the right-hand side does not contain 
$\tau$, we can apply the classical expansions for local probabilities in the central limit theorem, see Chapter VI in Petrov~\cite{Petrov}. To analyse the second probability on the right-hand side of \eqref{eq:reflection} we apply the strong Markov property to get 
\begin{equation}
\label{eq:tau-decomp}
\pr(S_n=x,\tau< n)=
\sum_{k=1}^{n-1}\sum_{y=0}^\infty\pr(S_k=-y,\tau=k)
\pr(S_{n-k}=x+y).
\end{equation}
To deal with the summands, we need to derive expansions for $\pr(S_k=-y,\tau=k)$. It turns out that such expansions can be derived from \eqref{eq:locprobtau}. Thus, we use the known results on probabilities of lower deviations ($x\ll\sqrt{n}$) to study the probabilities of normal deviations ($x\approx\sqrt{n}$).

We believe that one can also use Nagaev's method to show the existence of appropriate polynomials $P_j$.
But his approach is purely qualitative: it allows one to prove the existence of those polynomials but it remains unclear how to compute their coefficients. Our approach allows one to compute the coefficients. We illustrate this by the following result, where the polynomials $P_2$ and $P_3$ are computed.
\begin{corollary}
\label{cor:r=2}
Assume that the conditions of Theorem~\ref{thm:lattice} hold with $r=2$. Then
\begin{align*}
&P_2(t)=\frac{2\theta_0}{\sigma}t,\\
&P_3(t)=\frac{\theta_0m_3}{3\sigma^3}(t^4-5t^2+2)
+\frac{2\theta_1}{\sigma}(t^2-1),
\end{align*}
where $m_3=\e X_1^3$ and 
\begin{align*}
&\theta_0=\sum_{u=1}^\infty U_1(u)\pr(X_1<-u),\\
&\theta_1=\sum_{u=1}^\infty U_1(u)\e[(-X_1-u);X_1<-u].
\end{align*}
\end{corollary}
In the proof of Corollary~\ref{cor:r=2} we also show that the constants $\theta_0$ and $\theta_1$ can be expressed in the following way:
\begin{equation}
\label{eq:theta0}
\theta_0=\lim_{k\to\infty} k^{3/2}\pr(\tau=k)
\end{equation}
and
\begin{equation}
\label{eq:theta1}
\theta_1=\lim_{k\to\infty} k^{3/2}\e[-S_\tau;\tau=k].
\end{equation}

One can use our asymptotic expansions to obtain a rate of convergence in the limit theorem for random walks conditioned to stay positive.
Assume that the conditions of Theorem~\ref{thm:lattice} hold with $r=1$.
Then we have
$$
\pr(S_n=x,\tau>n)=2\theta_0\frac{x}{\sigma^2n^{3/2}}
e^{-\frac{x^2}{2\sigma^2n}} 
+o(n^{-3/2})
+O\left(\frac{1}{\min\{n^{2},x^4\}}\right).
$$
Then, uniformly in $x\in[u\sigma\sqrt{n},v\sigma\sqrt{n}]$,
$$
\pr(S_n=x,\tau>n)=2\theta_0\frac{x}{\sigma^2n^{3/2}}
e^{-\frac{x^2}{2\sigma^2n}} 
+O\left(n^{-3/2}\right),
$$
where $0<u<v$ are fixed numbers. Consequently,
\begin{align*}
    \pr
    \left(\frac{S_n}{\sigma\sqrt{n}}\in[u,v],\tau>n\right)
&=
    \sum_{x\in[u\sigma\sqrt{n},v\sigma\sqrt{n}]}
        \pr(S_n=x,\tau>n)\\
&=
    \frac{2\theta_0}{\sqrt{n}}
    \sum_{x\in[u\sigma\sqrt{n},v\sigma\sqrt{n}]}
        \frac{x}{\sigma\sqrt{n}}
        e^{-\frac{x^2}{2\sigma^2n}} 
        \frac{1}{\sigma\sqrt{n}}
+O\left(n^{-1}\right).
\end{align*}
Approximating the sum by the integral, we conclude that 
\begin{equation}
\label{eq:BE1}
    \pr\left(\frac{S_n}{\sigma\sqrt{n}}\in[u,v],\tau>n\right)
=
    \frac{2\theta_0}{\sqrt{n}}\left(e^{-u^2/2}-e^{-v^2/2}\right)
+
    O\left(n^{-1}\right).
\end{equation}
According to Theorem 1 in \cite{DTW23},
$$
\pr(\tau>n)=\frac{2\theta_0}{\sqrt{n}}+O(n^{-1}).
$$
Combining this with \eqref{eq:BE1}, we get
\begin{equation}
\label{eq:BE2}
\pr\left(\frac{S_n}{\sigma\sqrt{n}}\in[u,v]\,\Big|\tau>n\right)
=e^{-u^2/2}-e^{-v^2/2}+O(n^{-1/2}).
\end{equation}
Since the remainder term depends on $u$ and $v$, this estimate is much weaker than the standard Berry-Esseen inequality. Furthermore, \eqref{eq:BE2} is derived under the assumption that the fourth moment is finite.  
A stronger version of~\eqref{eq:BE2}, which is an analogue of the Berry-Esseen inequality, has been proved in~\cite{DTW24}: if $\E|X_1|^3$ is finite then there exists a constant $C$ such that 
\begin{equation}
\label{eq:BE3}
\sup_{u\ge0}\left|
\pr\left(\frac{S_n}{\sigma\sqrt{n}}>u\,\Big|\tau>n\right)
-e^{-u^2/2}\right|\le\frac{C (\E|X_1|^3)^3}{\sigma^9\sqrt{n}}.
\end{equation}
A slower rate of convergence for random walks with an arbitrary starting point has recently been obtained by Grama and Xiao~\cite{GX}. 

Let 
$$
\overline{\tau}:=\inf\{n\ge1: S_n<0\}.
$$
Since in \cite{DTW23} we have also derived the decomposition
\begin{align} 
\label{eq:locprobtau2}
    \pr(S_n=x,\overline{\tau}>n)
=
    \sum_{j=1}^{\lf \frac{r+2}{2}\rf} 
        \frac{\overline{U}_j(x)}{(n+1)^{j+1/2}}
+
    O\left(
        \frac{(1+x)^{r+2}}{n^{(r+4)/2}}\log^{{\lf \frac{r+1}{2}\rf}}n
    \right),
\end{align}
we can repeat the proof of Theorem~\ref{thm:lattice}. As a result we get
\begin{theorem}
\label{thm:lattice2}
Assume that $\e|X_1|^{r+3}$ is finite for some integer $r \ge 0$.
Assume also that $S_n$ is integer-valued with maximal span $1$. Then there exist polynomials $\overline{P}_2$, $\overline{P}_3$, \ldots, $\overline{P}_{r+2}$ such that
\begin{align*}
    \pr(S_n=x,\overline{\tau}>n) 
&=
    e^{-\frac{x^2}{2\sigma^2n}}
    \sum_{\nu = 2}^{r+2}
        \frac{\overline{P}_\nu(x/\sigma\sqrt{n})}{n^{\nu/2}}\\ 
&\hspace{2cm}
+
    o\left(
        \frac{1}{n^{(r+2)/2}}
    \right)
+
    O\left( \frac{1}{\min (\sqrt{n}, |x|)^{r+3}} \right).
\end{align*}
\end{theorem}
The only change we need is a new version of Lemma~\ref{lem:sumovery}, see Remark~\ref{rem:lemma18} below.

One term of an Edgeworth expansion for 
global probabilities $\pr_b(\tau\le n, S_n\le x)$ 
for a random walk started at a point $b = b(n)$ 
was obtained in~\cite{SiegmundYuh}. This result was later used in~\cite{BroadieGlassermanKou} to obtain the first order correction for the price of a discrete-time barrier option. 
We plan to obtain the full asymptotic expansion for normal deviations when the random walk starts at an arbitrary point. This expansion should allow us to extend the first order correction in~\cite{BroadieGlassermanKou} to obtain further corrections for discrete-time barrier options.

The remaining part of the paper is organized as follows.
In many places in the proofs we will need to approximate sums by integrals; all bounds and asymptotic relations of that kind are collected in the next section. 
In Section 3, we present an expansion for 
$\pr(S_n=x)$
 that is well-known in the literature, in a form that is useful for our purposes.
 Section 4 is devoted to the main part of the proof, where we derive an expansion for $\pr(S_n=x,\tau<n)$ which is valid for all $x$. The resulting expansion contains, from the formal point of view, some non-polynomial terms. More precisely, the obtained expansion contains negative powers of $x/\sqrt{n}$ as prefactors of $e^{-x^2/2\sigma^2n}$. 
 In Section 5 we show that non-polynomial terms have zero coefficients and simplify the expansion obtained in Section 4. These simplifications also lead to the algorithmic computability of the coefficients in the expansions of Theorem~\ref{thm:lattice}. This is illustrated in Section 6, where we prove Corollary~\ref{cor:r=2} and compute polynomials $P_2$ and $P_3$.
\section{Approximation of sums by integrals}
In this section we collect estimates for sums of values of elementary functions at integer points. The need for such estimates is quite clear from the decomposition \eqref{eq:tau-decomp}, which is the starting point of our approach.
\begin{lemma} \label{lem:firsthalfsum}
Let $a\ge 0,b>1$, $t\in [0,1]$ and $N_n \in [\frac{n}{3}, \frac{2n}{3}-1]$. Then there exists a constant $\gamma_{a,b}$ depending only on $a$ and $b$, such that
\begin{align*}
    \sum_{k=1}^{N_n} 
    \frac{1}{(n-k-t)^{a}(k+t)^{b}}
    e^{-\frac{z^2}{2 k}}
    \le
    \frac{\gamma_{a,b}}{n^a |z|^{2(b-1)}}.
\end{align*}
\end{lemma}
\begin{proof}
Noting that $n-k-t \ge n/3$ for all $k\le N_n$, we have
\begin{align*}
    \sum_{k=1}^{N_n}
        \frac{1}{(n-k-t)^{a}(k+t)^{b}}
        e^{-\frac{z^2}{2k}}
\le
    \frac{3^{a}}{n^{a}}
    \sum_{k=1}^{\infty}
        \frac{1}{(k+t)^{b}}
        e^{-\frac{z^2}{2 (k+t)}}.
\end{align*}
For all $k\ge 1$ and any $\delta \in [0,1]$ one has
\begin{align*}
    \frac{1}{(k+t)^b}e^{-\frac{z^2}{2 (k+t)}}
    \le
    \frac{1}{(k+t)^{b}}e^{-\frac{z^2}{2 (k+t+\delta)}}
    \le
    \frac{2^{b}}{(k+t+\delta)^b}e^{-\frac{z^2}{2 (k+t+\delta)}}.
\end{align*}
Therefore,
$$
\frac{1}{(k+t)^b}e^{-\frac{z^2}{2 (k+t)}}
\le 2^b
\int_{k+t}^{k+t+1}y^{-b}
e^{-\frac{z^2}{2 y}}dy
$$
and 
\begin{align*}
    \sum_{k=1}^{\infty}\frac{1}{(k+t)^{b}}e^{-\frac{z^2}{2 (k+t)}} \le
    2^{b}\int^{\infty}_{1+t} y^{-b}e^{-\frac{z^2}{2 y}}dy
    \le
     2^{b}\int^{\infty}_0 y^{-b}e^{-\frac{z^2}{2 y}}dy.
\end{align*}
Substituting $s = \frac{z^2}{2y}$ we obtain
\begin{align*}
    \int^{\infty}_0 y^{-b}e^{-\frac{z^2}{2 y}}dy = \frac{2^{b-1}}{|z|^{2(b-1)}}\int_0^\infty s^{b-2}e^{-s}ds = \Gamma(b-1)\frac{2^{b-1}}{|z|^{2(b-1)}}.
\end{align*}
As a result we have
\begin{align*}
    \sum_{k=1}^{N_n}
        \frac{1}{(n-k-t)^{a}(k+t)^{b}}
        e^{-\frac{z^2}{2 k}} 
\le
    \Gamma(b-1)
    \frac{3^{a}2^{2b-1}}{n^{a}|z|^{2(b-1)}}.
\end{align*}
Thus, the proof is complete.
\end{proof}
\begin{lemma}\label{lem:secondhalfsum}
Consider a sequence $\{v_k\}$ such that
\begin{align}
\label{eq:v-cond}
    |v_1| \le C, \quad |v_k| \le \frac{C\log^d k}{k^a},\quad k\ge2
\end{align}
for some $d\ge0$ and $a>0$.
Then there exist polynomials $K_\nu$ of degree $\nu$ such that,
for all $b>0$ and all $m\in\{0,1,2,\ldots\}$, the following asymptotic expansions hold uniformly in $z$.
\begin{itemize}
\item[(1)] If $a>1$ then
\begin{align*}
    \sum_{k=1}^{\lf n/2\rf} 
        \frac{z^m v_k}{(n-k)^{b}}e^{-\frac{z^2}{2 (n-k)}}
=
    \frac{z^m}{n^b} 
    e^{-\frac{z^2}{2 n}}
    \sum_{\nu=0}^{\lf a \rf-1}
            \frac{K_{\nu}(z^2/n)}{n^{\nu}}
+
    O\left(\frac{\log^{d+\delta_a} n}{n^{a+b-m/2-1}}\right),
\end{align*}
where $\delta_a:={\rm 1}\{a\in\N\}$.
\item[(2)] If $a\in(0,1]$ then 
\begin{align*}
    \sum_{k=1}^{\lf n/2\rf} 
        \frac{z^m v_k}{(n-k)^{b}}e^{-\frac{z^2}{2 (n-k)}}
=
    O\left(\frac{\log^{d+\delta_a} n}{n^{a+b-m/2-1}}\right).
\end{align*}
\end{itemize}
\end{lemma}
\begin{proof}
If $a\le 1$ then, using the fact that the function $t^me^{-t^2/2}$ is bounded, we obtain
\begin{align*}
&\sum_{k=1}^{\lf n/2\rf} 
        \frac{z^m |v_k|}{(n-k)^{b}}e^{-\frac{z^2}{2 (n-k)}}
\le \frac{z^m}{(n/2)^b} e^{-\frac{z^2}{2n}} 
         \sum_{k=1}^{\lf n/2\rf} |v_k| 
         \le \frac{C}{n^{b-m/2}}\sum_{k=1}^{n} |v_k|.
\end{align*}
Using now our assumption on the sequence $\{v_k\}$, we obtain the second claim of the lemma.

We now turn to the case $a>1$.
We first notice that for every $j\ge1$,  there exists a polynomial $h_j(t)$ of degree $j$ such that 
$$
\frac{d^j}{d t^j}e^{-t^2/2}=h_j(t)e^{-t^2/2}.
$$
Note that the polynomial
$h_j(t)$ differs from the $j$-th Hermite polynomial only by a factor $(-1)^j$. 
Using the standard Taylor expansion for $e^x$, we have 
$$
e^{-\frac{t^2}{2}}
=\sum_{j=0}^\infty\frac{(-1)^j}{j!2^j}t^{2j}.
$$
This implies that
\begin{align} \label{eq:mucoefficients}
    h_{2\mu}(0) = (-1)^\mu \frac{(2\mu)!}{2^\mu\mu!}
    \quad\text{and}\quad 
    h_{2\mu+1}(0)=0
\end{align}
for all $\mu\ge0$.

Applying now the Taylor formula to the function
$t\mapsto e^{-t^2/2}$, we have 
\begin{align}
\label{eq:exp-taylor}
\nonumber
e^{-t^2/2}
&=\sum_{j=0}^{R-1}\frac{h_j(0)}{j!}t^j
+\frac{h_R(\theta)}{R!}e^{-\theta^2/2}t^R\\
&=\sum_{\nu=0}^{\lfloor \frac{R-1}{2}\rfloor}\frac{h_{2\nu}(0)}{(2\nu)!}t^{2\nu}+\frac{h_R(\theta)}{R!}e^{-\theta^2/2}t^R,
\quad \theta=\theta(t)\in(0,t)
\end{align}
for every $R\ge1$. 

One has the equality 
\begin{align*}
    e^{-\frac{z^2}{2 (n-k)}} 
=
    e^{-\frac{z^2}{2 n}} 
    e^{-\frac{k z^2}{2 n(n-k)}}.
\end{align*}
We apply \eqref{eq:exp-taylor} to the second exponential factor with
\begin{align*}
    R_a
=
    \begin{cases}
    2\lf a\rf, \quad \; \; \; a \notin \mathbb{N},\\
        2a-1, \quad a\in \mathbb{N}
    \end{cases}
\end{align*}
to get 
\begin{align} \label{eq:prop:sechalf:expdecomp}
    e^{
    \frac{-k z^2}
    {2 n(n-k)}
    }
=
    \sum_{\nu=0}^{\lf a \rf - 1}
        \frac{h_{2\nu}(0)}{(2\nu)!}
        \frac{ k^\nu z^{2\nu}}{n^\nu (n-k)^\nu}
+
\frac{h_{R_a}(\theta)}{R_a!}e^{-\theta^2/2}
\frac{
    k^{R_a/2}|z|^{R_a}
}{n^{R_a/2}(n-k)^{R_a/2}
}.
\end{align}
Here we have used the fact that 
$\lf \frac{R_a-1}{2} \rf = \lf a \rf -1$ for all $a$.
In view of \eqref{eq:v-cond},
\begin{align}
\label{eq:sum-bound}
    \sum_{k=1}^{\lf \frac{n}{2}\rf}
        |v_k| k^{R_a/2} 
\le 
    C
    \left(
        1
    +
        \sum_{k=2}^{\lf \frac{n}{2}\rf}
            \frac{\log^d k}{k^{a-R_a/2}}
    \right)
\le C_1\left\{
\begin{array}{ll}
n^{1-\{a\}}\log^{d} n, &a\notin\N,\\
n^{1/2}\log^{d} n, &a\in\N.
\end{array}
\right.
\end{align}
Note next that
\begin{align} \label{eq:boundedfunctions}
    |h_{R_a}(\theta)| e^{-\theta^2/2}\le C(a)
\quad
\text{and}
\quad
    \left(\frac{|z|}{\sqrt{n}}\right)^{r}e^{-z^2/2n}\le \widetilde C(r)
\end{align}
with some constants $C(a)$ and $\widetilde C(r)$. Since also $n-k \ge n/2$, 
in view of~\eqref{eq:boundedfunctions}, 
we conclude that 
the remainder in \eqref{eq:prop:sechalf:expdecomp} satisfies
\begin{align*}
&\left|
\frac{h_{R_a}(\theta)}{R_a!}e^{-\theta^2/2}
        e^{-\frac{z^2}{2 n}}
        \sum_{k=1}^{\lf \frac{n}{2}\rf}
          \frac{v_k}{(n-k)^b}
            \frac{k^{R_a/2}|z|^{R_a+m}}{n^{R_a/2}(n-k)^{R_a/2}}
    \right|\\
 &\hspace{1cm}
 \le \frac{C(a)}{R_a!}
            \frac{|z|^{R_a+m}}{n^{R_a/2}(n/2)^{b+R_a/2}}
        e^{-\frac{z^2}{2 n}}
        \sum_{k=1}^{\lf \frac{n}{2}\rf}
          |v_k|k^{R_a/2}
    \\
&\hspace{1cm}
\le 2^{b+R_a/2}\frac{C(a)\widetilde C(R_a+m)}{n^{b+R_a/2-m/2}}
\sum_{k=1}^{\lf \frac{n}{2}\rf}
|v_k| k^{R_a/2}
\le C_3
    \frac{\log^{d}n}{n^{a+b-m/2-1}}.
\end{align*}
Consequently, we obtain from the last bound and~\eqref{eq:prop:sechalf:expdecomp}
that 
\begin{multline} \label{eq:prop:repres1}
z^m
\sum_{k=1}^{\lf n/2\rf} 
\frac{v_k}{(n-k)^{b}}e^{-\frac{z^2}{2 (n-k)}}
=
z^m
e^{-\frac{z^2}{2 n}}
\sum_{\nu=0}^{\lf a\rf-1}
\sum_{k=1}^{\lf n/2\rf} 
\frac{h_{2\nu}(0)}{(2\nu)!} 
            \frac{v_k k^{\nu}z^{2\nu}}{n^{\nu}(n-k)^{b+\nu}}
\\
+
    O\left(\frac{\log^{d} n}{n^{a+b-m/2-1}}\right).
\end{multline}
Next, applying the Taylor formula to the function
$(n-k)^{-(b+\nu)}$, we have, for some $\psi_k\in(0,k)$,
\begin{align}\label{eq:taylor1}
\nonumber
\frac{1}{(n-k)^{b+\nu}}
&=   
\sum_{i=0}^{M_a-\nu-1}
\frac{\prod_{j=0}^{i-1}(b+\nu+j)}{i!}
\frac{k^i }{n^{b+\nu+i}}\\
&\hspace{3cm}
+
\frac{\prod_{j=0}^{M_a-\nu}(b+\nu+j)}{(M_a-\nu)!}
\frac{k^{M_a-\nu}}{(n-\psi_k)^{b+M_a}},
\end{align}
where  
$$
M_a:=\left\{ 
\begin{array}{ll}
\lfloor a \rfloor, &a\notin\N,\\
a-1, &a\in\N.
\end{array}
\right.
$$
Using the second inequality in \eqref{eq:boundedfunctions} with 
$r=2\nu+m$
, we infer that the sum of remainder terms can be bounded in the following way:
\begin{align} \label{eq:prop:estim1}
\left|\frac{z^{2\nu+m}}{n^\nu}
e^{-\frac{z^2}{2 n}}
\sum_{k=1}^{\lf \frac{n}{2}\rf}
\frac{v_k k^{M_a}}{(n-\psi_k)^{b+M_a}}
    \right|
\le
    \frac{C}{n^{b+M_a-m/2}}
      \sum_{k=1}^{\lf \frac{n}{2}\rf}|v_k| k^{M_a}.
\end{align}
Similar to \eqref{eq:sum-bound} we get 
\begin{align*}
\sum_{k=1}^{\lf \frac{n}{2}\rf}
|v_k| k^{M_a} 
\le C
    \left(
        1
    +
        \sum_{k=2}^{\lf \frac{n}{2}\rf}\frac{\log^d k}{k^{a-M_a}}
    \right)
\le C_1\left\{
\begin{array}{ll}
n^{1-\{a\}}\log^{d} n, &a\notin\N,\\
\log^{d+1} n, &a\in\N.
\end{array}
\right.
\end{align*}
Plugging this into \eqref{eq:prop:estim1}, we obtain 
\begin{align} 
\label{eq:prop:estim11}
\left|\frac{z^{2\nu+m}}{n^\nu}
e^{-\frac{z^2}{2 n}}
\sum_{k=1}^{\lf \frac{n}{2}\rf}
\frac{v_k k^{M_a}}{(n-\psi_k)^{b+M_a}}
    \right|
\le
C\frac{\log^{d+\delta_a}n}{n^{b+a-m/2-1}}.
\end{align}
Plugging \eqref{eq:taylor1} into \eqref{eq:prop:repres1} and applying \eqref{eq:prop:estim11}, we conclude that 

\begin{multline*}
    z^m
    \sum_{k=1}^{\lf n/2\rf} 
        \frac{v_k}{(n-k)^{b}}e^{-\frac{z^2}{2 (n-k)}}
\\
=
    z^m
    e^{-\frac{z^2}{2 n}}
    \sum_{\nu=0}^{\lf a\rf-1}
    \frac{h_{2\nu}(0)}{(2\nu)!}
    \frac{z^{2\nu}}{n^{\nu}}
    \sum_{i=0}^{M_a-\nu-1}
    \frac{\prod_{j=0}^{i-1}(b+\nu+j)}{i!n^{b+\nu+i}}
\sum_{k=1}^{\lf n/2\rf} 
                v_k k^{\nu+i}
\\
+
    O\left(\frac{\log^{d+\delta_a} n}{n^{a+b-m/2-1}}\right).
\end{multline*}
For every $\lambda <a-1$ we infer from \eqref{eq:v-cond} that
\begin{align} \label{eq:propsummation}
    \sum_{k=1}^{\lf n/2 \rf}
        v_k k^{\lambda}
=
    \sum_{k=1}^{\infty}
        v_k k^{\lambda}
-
    \sum_{k > \lf \frac{n}{2}\rf}^\infty
        v_k k^{\lambda}
=
    g_{\lambda} 
+
    O\left(\frac{\log^d n}{n^{a-\lambda-1}}\right),
\end{align}
where 
$$
g_\lambda:=\sum_{k=1}^\infty k^\lambda v_k.
$$
This implies that 
\begin{multline*}
    z^m
    \sum_{k=1}^{\lf n/2\rf} 
        \frac{v_k}{(n-k)^{b}}e^{-\frac{z^2}{2 (n-k)}}
\\
=
    z^m
    e^{-\frac{z^2}{2 n}}
    \sum_{\nu=0}^{\lf a\rf-1}
    \frac{h_{2\nu}(0)}{(2\nu)!}
    \frac{z^{2\nu}}{n^{\nu}}
    \sum_{i=0}^{M_a-\nu-1}
    \frac{\prod_{j=0}^{i-1}(b+\nu+j)}{i!n^{b+\nu+i}}
g_{\nu+i}
+O\left(\frac{\log^{d+\delta_a} n}{n^{a+b-m/2-1}}\right).
\end{multline*}
Gathering the terms with $\nu + i=\mu$,  we conclude that

\begin{align*}
    \sum_{k=1}^{\lf n/2\rf} 
        \frac{z^m v_k}{(n-k)^{b}}e^{-\frac{z^2}{2 (n-k)}}
=
    \frac{z^m}{n^b} 
    e^{-\frac{z^2}{2 n}}
    \sum_{\mu=0}^{M_a-1}
            \frac{K_{\mu}(z^2/n)}{n^{\mu}}
+
    O\left(\frac{\log^{d+\delta_a} n}{n^{a+b-m/2-1}}\right)
\end{align*}
with some polynomials $K_\mu$ of degree $\mu$. Noting that  
$M_a = \lceil a \rceil -1\le \lfloor a \rfloor$  for all $a$, we may extend the summation region to $\{0,1,\ldots,\lfloor a \rfloor-1\}$ by adding, if needed, zero polynomials.
Thus, the lemma is proved.

\end{proof}
\begin{lemma}\label{lem:secondhalfsumgeneral}
For all $a, b>0$, $a\notin \mathbb{N}$ and $t \in (0,1)$ there exist numbers $c_{i,\nu}$ such that, uniformly in $z$, 
we have the following equalities.
\begin{itemize}
\item[(1)] If $a>1$ then
\begin{multline*}
    \sum_{k=2}^{\lf n/2\rf} 
        \frac{z^m}{(n-k+t)^{b}(k-t)^a}
        e^{-\frac{z^2}{2 (n-k+t)}}
\\
=
    \frac{z^m}{n^b}
    e^{-\frac{z^2}{2 n}}
    \sum_{\nu=0}^{\lf a \rf-1}
        \frac{z^{2\nu}}{n^{\nu}}
        \sum_{i=\nu}^{\lf a \rf-1}
            \frac{c_{i,\nu}g_{a-i}(t)}{n^{i}}
+
    O\left(\frac{1}{n^{a+b-m/2-1}}\right),
\end{multline*}
where
\begin{align*}
    g_{\lambda}(t) 
:=
    \sum_{k=2}^{\infty}
        \frac{1}{(k-t)^{\lambda}}\quad\text{for every } \lambda>1.
\end{align*}
\item[(2)] If $a\in(0,1)$ then
\begin{align*}
    \sum_{k=2}^{\lf n/2\rf} 
        \frac{z^m}{(n-k+t)^{b}(k-t)^a}
        e^{-\frac{z^2}{2 (n-k+t)}}
=
    O\left(\frac{1}{n^{a+b-m/2-1}}\right).
\end{align*}
\end{itemize}
\end{lemma}
\begin{proof}
In the case $a<1$ we have 
\begin{align*}
 \sum_{k=2}^{\lf n/2\rf} 
        \frac{z^m}{(n-k+t)^{b}(k-t)^a}
        e^{-\frac{z^2}{2 (n-k+t)}}
&\le \frac{z^m e^{-\frac{z^2}{{2n}}}}{(n/2)^b}
\sum_{k=2}^{\lf n/2\rf} \frac{1}{(k-t)^a}\\
&\le C n^{m/2-b}\sum_{k=1}^n k^{-a}=O(n^{m/2+1-a-b}). 
\end{align*}
Thus, the second claim of the lemma is proved.

Assume now that $a>1$.
Set, for brevity, $s_k:=k-t$. Similar to the proof of the previous lemma, we write
\begin{align*}
    e^{-\frac{z^2}{2 (n-s_k)}} = e^{-\frac{z^2}{2 n}} e^{-\frac{s_k z^2}{2 n(n-s_k)}}.
\end{align*}
Using \eqref{eq:exp-taylor}, we have 
\begin{align} \label{eq:prop:sechalf:expdecomp1}
    e^{
    \frac{-s_k z^2}
    {2 n(n-s_k)}
    }
=
    \sum_{\nu=0}^{\lf \frac{R_a-1}{2}\rf}
        \frac{h_{2\nu}(0)}{(2\nu)!}
        \frac{ s_k^\nu z^{2\nu}}{n^\nu (n-s_k)^\nu}
+
\frac{h_{R_a}(\theta)}{R_a!}e^{-\theta^2/2}
\frac{s_k^{R_a/2}|z|^{R_a}}{n^{R_a/2}(n-s_k)^{R_a/2}}.
\end{align}
Recall that the assumption $a\notin\N$ implies that 
$R_a=2\lf a\rf$. Noting that $n-s_k\ge n/2$ and  using the second inequality in \eqref{eq:boundedfunctions}, we conclude that
\begin{align*}
&\left|
\frac{h_{R_a}(\theta)}{R_a!}e^{-\theta^2/2}
        e^{-\frac{z^2}{2 n}}
        \sum_{k=2}^{\lf \frac{n}{2}\rf}
            \frac{1}{(n-s_k)^b s_k^a}
            \frac{s_k^{\lf a \rf}|z|^{2\lf a \rf+m}}{n^{\lf a \rf}(n-s_k)^{\lf a \rf}}
    \right|\\
&\hspace{2cm}
\le
    e^{-\frac{z^2}{2 n}}
    \left(
        \frac{z^2}{n}
    \right)^{\lf a \rf +m/2}
    \frac{C}{n^{\lf a \rf+b - m/2}}
    \sum_{k=2}^{\lf \frac{n}{2} \rf}
        s_k^{-\{a\}}\\
&\hspace{2cm}
\le
    \frac{C}{n^{\lf a \rf+b - m/2}}
    \sum_{k=2}^{\lf \frac{n}{2} \rf}
        s_k^{-\{a\}}
\le
    \frac{C}{n^{a+b-m/2-1}}.
\end{align*}
Therefore,
\begin{align} 
\label{eq:prop:repres11}
\nonumber
&\sum_{k=2}^{\lf n/2\rf} 
        \frac{z^m}{(n-s_k)^{b}s_k^a}e^{-\frac{z^2}{2 (n-s_k)}}\\
&\hspace{1cm}=
    z^m
    e^{-\frac{z^2}{2 n}}
    \sum_{\nu=0}^{\lf a \rf-1}
        \sum_{k=2}^{\lf n/2\rf} \frac{h_{2\nu}(0)}{(2\nu)!} 
            \frac{z^{2\nu}}{n^{\nu}(n-s_k)^{b+\nu}s_k^{a-\nu}}
+
    O\left(\frac{1}{n^{a+b-m/2-1}}\right).
\end{align}
Applying \eqref{eq:taylor1} and noting that
$M_a=\lf a\rf$ for $a\notin\N$, we have
\begin{align}\label{eq:taylor11}
\nonumber
\frac{1}{(n-s_k)^{b+\nu}}
&=   
\sum_{i=0}^{\lf a\rf-\nu-1}
\frac{\prod_{j=0}^{i-1}(b+\nu+j)}{i!}
\frac{s_k^i }{n^{b+\nu+i}}\\
&\hspace{3cm}
+
\frac{\prod_{j=0}^{\lf a\rf-\nu}(b+\nu+j)}{(\lf a\rf-\nu)!}
\frac{s_k^{\lf a\rf-\nu}}{(n-\psi_k)^{b+\lf a\rf}},
\end{align}
for some $\psi_k \in [0,s_k]$. Similar to the proof of the previous lemma we get, uniformly in $\theta\in(0,1)$,
\begin{align} \label{eq:prop:estim12}
    \left|
        \frac{z^{2\nu+m}}{n^\nu}
        e^{-\frac{z^2}{2 n}}
        \sum_{k=2}^{\lf \frac{n}{2}\rf}
            \frac{1}{(n-\theta)^{b+\lf a \rf}s_k^{\{a\}}}
    \right|
\le
    \frac{C}{n^{a+b-m/2-1}}.
\end{align}
Plugging \eqref{eq:taylor11} into \eqref{eq:prop:repres11} and  applying then \eqref{eq:prop:estim12}, we obtain
\begin{multline*}
    \sum_{k=2}^{\lf n/2\rf} 
        \frac{z^m}{(n-s_k)^{b}s_k^a}
        e^{-\frac{z^2}{2 (n-s_k)}}
\\
=
    z^m
    e^{-\frac{z^2}{2 n}}
    \sum_{\nu=0}^{\lf a \rf-1}
        \frac{h_{2\nu}(0)}{(2\nu)!} 
        \frac{z^{2\nu}}{n^{\nu}}
        \sum_{i=0}^{\lf a \rf-\nu-1}
        \frac{\prod_{j=0}^{i-1}(b+\nu+j)}{i!n^{b+\nu+i}}
        \sum_{k=2}^{\lf n/2\rf} 
                \frac{1}{s_k^{a-\nu-i}}
\\
+
    O\left(\frac{1}{n^{a+b-m/2-1}}\right).
\end{multline*}
Recall that $s_k=k-t$. For every $\lambda >1$ we have, uniformly in $t \in [0,1]$,
\begin{align*}
    \sum_{k=2}^{\lf n/2 \rf}
        \frac{1}{(k-t)^{\lambda}}
=
    \sum_{k=2}^{\infty}
        \frac{1}{(k-t)^{\lambda}}
-
    \sum_{k > \lf \frac{n}{2}\rf}^\infty
        \frac{1}{(k-t)^{\lambda}}
=
    g_{\lambda}(t) 
+
    O\left(\frac{1}{n^{\lambda-1}}\right).
\end{align*}
The term produced by the remainder
$O\left(\frac{1}{n^{\lambda-1}}\right)$
can be estimated by applying the second inequality in \eqref{eq:boundedfunctions}.
This leads then to the equality
\begin{multline*}
    \sum_{k=2}^{\lf n/2\rf} 
        \frac{z^m}{(n-s_k)^{b}s_k^a}
        e^{-\frac{z^2}{2 (n-s_k)}}
=
    \frac{z^m}{n^b}
    e^{-\frac{z^2}{2 n}}
    \sum_{\nu=0}^{\lf a \rf-1}
        \frac{z^{2\nu}}{n^{\nu}}
        \sum_{i=\nu}^{\lf a \rf-1}
            \frac{c_{i,\nu}g_{a-i}(t)}{n^{i}}
\\
+
    O\left(\frac{1}{n^{a+b-m/2-1}}\right),
\end{multline*}
where $c_{i,\nu}$ are some constants.
Thus, the lemma is proved.
\end{proof}

\begin{lemma}\label{prop:integral}
For every integer $b\ge 0$ we have
\begin{align*}
    \int_{0}^1
    \frac{1}{\sqrt{1-u}}u^{-b-3/2}
    e^{-\frac{z^2}{2u}}du
=
    {\rm sgn}(z)
    \sqrt{2\pi} 
    e^{-\frac{z^2}{2}}
    \sum_{k=0}^{b} 
        \frac{(2k-1)!!}{z^{2k+1}}{b\choose k}.
\end{align*}
(Here we use the standard convention: $(-1)!!=1$.)
\end{lemma}
\begin{proof}
Substituting  $t=1/u$ we obtain 
\begin{align*}
\int_{0}^1
    \frac{1}{\sqrt{1-u}}u^{-b-3/2}
    e^{-\frac{z^2}{2u}}du 
    &= 
    \int_1^{\infty} \frac{1}{\sqrt{t-1}}t^b e^{-\frac{z^2}{2}t}dt\\
    &=
    e^{-\frac{z^2}{2}}
    \int_0^{\infty} x^{-1/2}(x+1)^b
    e^{-\frac{z^2}{2}x}dx.
\end{align*}
By the binomial theorem, 
\begin{align*}
\int_0^{\infty} x^{-1/2}(x+1)^b
    e^{-\frac{z^2}{2}x}dx
    &=\sum_{k=0}^b {b \choose k} \int_0^\infty x^{k-\frac{1}{2}}e^{-\frac{z^2}{2}x} dx \\ 
    &= \sum_{k=0}^b {b \choose k} 
    \left(\frac{z^2}{2}\right)^{-k-\frac12}
    \Gamma\left(k+\frac12\right).
\end{align*}
Noting that 
\[
\Gamma\left(k+\frac12\right) 
=\Gamma\left(\frac12\right) \prod_{j=1}^{k} \left(j-\frac{1}{2}\right)
= \sqrt{\pi} 2^{-k} (2k-1)!!
\]
we arrive at the conclusion. 
\end{proof}
The next result is the most important technical tool in our approach: we derive asymptotic expansions for sums containing combinations of polynomial and exponential functions. This allows us to deal with every term in asymptotic expansions for $\pr(S_n=x)$. 
\begin{proposition} \label{prop:basisconvolution}
There exist polynomials $P_{k,j,\ell,m}$ each of degree $k$ such that,
for all $r,m,j,\ell \ge 0$,
\begin{align}\label{prop:mainconv}
&\sum_{k=1}^{n-1}
        \frac{z^m}{k^{j+1/2}(n-k)^{\ell+1/2}}
        e^{-\frac{z^2}{2 (n-k)}}
=
    {\rm sgn}(z)\sqrt{2\pi}
    e^{-\frac{z^2}{2n}}
    \frac{Q_{j,\ell,m}(z/\sqrt{n})}{n^{\ell+j-m/2}}
        \\
\nonumber   
&\hspace{1cm}+
    e^{-\frac{z^2}{2n}}
    \frac{z^m}{n^{\ell+1/2}}
    \sum_{k=0}^{\lceil \frac{r+m}{2}\rceil - \ell+1}
        \frac{P_{k,j,\ell,m}(z^2/n)}{n^{k}}
+
    O\left(\frac{1}{(\min(\sqrt{n},1+|z|))^{r+3}}\right),
\end{align}
where
\begin{align*}
    Q_{j,\ell,m}(t)
=
    t^m
    \sum_{q=0}^{j}
        \gamma_{q,j,\ell} t^{2q}
        \sum_{k=0}^{\ell+j+q-1}
            \frac{(2k-1)!!
            \binom{\ell+j+q-1}{k}}
            {t^{2k+1}}
\end{align*}
and the reals $\gamma_{q,j,\ell}$ can be computed recursively; the corresponding expressions are given in \eqref{eq:gammas} below.
\end{proposition}
\begin{remark}
It is clear that some polynomial parts of the functions $Q_{j,\ell,m}$
can be included in $P_{k,j,\ell}$. We do not do that since the polynomials $P_{k,j,\ell}$ will not be important for our final result.
More precisely, we show later that only the functions $Q_{j,\ell,m}$
will appear in the asymptotic expansion in Theorem~\ref{thm:lattice}.\hfill$\diamond$
\end{remark}
\begin{proof}[Proof of Proposition~\ref{prop:basisconvolution}]
Let us define
\begin{align*}
    f_{j,\ell,m}(u) 
:=
    \frac{z^m}{(n-u)^{j+1/2}u^{\ell+1/2}}
    e^{-\frac{z^2}{2 u}}.
\end{align*}
Then
\begin{align} \label{eq:flmnotationrewriting}
    \sum_{k=1}^{n-1}
        \frac{z^m}{k^{j+1/2}(n-k)^{\ell+1/2}}
        e^{-\frac{z^2}{2 (n-k)}}
=
    \sum_{k=1}^{n-1}f_{j,\ell,m}(n-k)
=
    \sum_{k=1}^{n-1}f_{j,\ell,m}(k).
\end{align}
Notice that, for all $a\ge0$ and $\gamma>0$, there exists a constant 
$C(a,\gamma)$ such that
\begin{align}
\label{eq:simple-ineq}
    e^{-\gamma\frac{{z}^2}{n}}\frac{|z|^{a}}{n^{a/2}} \le C(a,\gamma).
\end{align}
We have for every $m\ge0$ and every $d\ge 1$, uniformly in $z$,
\begin{align}
\label{eq:another-expan}
\nonumber
    z^m
    e^{-\frac{z^2}{2(n-1)}} 
=&
    z^m
    e^{-\frac{z^2}{2n}} 
    e^{-\frac{z^2}{2(n-1)n}} 
\\
=&
    z^m
    e^{-\frac{z^2}{2n}}
    \left(
        \sum_{i=0}^{d-1} 
            \left(\frac{z^2}{n}\right)^{i}
            \frac{(-1)^i}{i! (2 (n-1))^i}
    \right)
+
    O\left( \frac{1}{n^{d-m/2}}\right).
\end{align}
Here we have used the Taylor formula and \eqref{eq:simple-ineq} to estimate the remainder.

Furthermore, by the Taylor formula,
$$
(n-1)^{-t-1/2}=n^{-t-1/2}+c_1(t)n^{-t-1-1/2}+
\ldots+c_s(t)n^{-t-s-1/2}+O(n^{-t-s-3/2})
$$
for all $t,s\ge0$. Combining this with \eqref{eq:another-expan},
we conclude that
there exist polynomials $P^{(0)}_{q,j,\ell}$ such that, uniformly in $z$,
\begin{multline} \label{eq:fderivative2}
    f_{j,\ell,m}(n-1) 
=
    \frac{z^m}{(n-1)^{\ell+1/2}} 
    e^{-\frac{z^2}{2(n-1)}} 
\\
=
    e^{-\frac{z^2}{2n}}
    \frac{z^m}{n^{\ell+1/2}}
    \sum_{q=0}^{d-1} 
        \frac{P^{(0)}_{q, j, \ell}(z^2/n)}{n^{q}}
+
    O\left( \frac{1}{n^{d+\ell - m/2 +1/2}}\right).
\end{multline}
Choosing $d$ large enough we obtain the desired decomposition for
$f_{j,\ell,m}(n-1)$.

We next consider the sum on the right-hand side of \eqref{eq:flmnotationrewriting} from $k=1$ to $n-2$.
By the Euler-Maclaurin formula (see, for example, Gel'fond's book \cite{Gelfond}), for every $p\ge1$,
\begin{multline}
\label{eq:EM}
    \sum_{k=1}^{n-2} f_{j,\ell,m}(k)
=
    \int_{1}^{n-1} 
        f_{j,\ell,m}(t)dt 
\;
+
    \sum_{\nu=1}^{p-1} 
        \frac{B_\nu}{\nu!}\big(f_{j,\ell,m}^{(\nu-1)}(n-1) - f_{j,\ell,m}^{(\nu-1)}(1)\big) 
\\
-
    \frac{1}{p!} 
    \int_0^1 
        \big[ B_p(1-t) - B_p \big]
        \sum_{k=1}^{n-2} 
            f_{j,\ell,m}^{(p)}(k+t)dt,
\end{multline}
where $B_\nu$ and $B_p(t)$ are Bernoulli numbers and Bernoulli polynomials respectively.

Due to Lemma~\ref{lem:eulerintegralrest} below,  there exist polynomials $P_{i,j,\ell}$ of degree $i$ such that
\begin{multline} 
\label{eq:eulerintegralrest2}
    \int_0^1 \big[B_p(u) - B_p\big] 
        \sum_{k = 1}^{n-2} 
            f^{(p)}_{j,\ell,m}(k+1-u)du
\\
=
    e^{-\frac{z^2}{2n}} 
    \frac{z^m}{n^{\ell+1/2}}
    \sum_{i=0}^{j+p-1}
        \frac{P_{i,j,\ell}(z^2/n)}{n^{i}} 
+
O\left(\frac{1}{\min (\sqrt{n},1+|z|)^{2(j+\ell+p)-m}}\right).
\end{multline}
We next consider the second summand on the right-hand side of \eqref{eq:EM}.

It is immediate from the definition of $f_{j,\ell,m}$
that 
\begin{align} \label{eq:fderivative}
    \frac{d}{du} f_{j,\ell,m}
=
    \frac{1}{2} f_{j,\ell+2,m+2} 
+
    (j+1/2)f_{j+1,\ell,m} 
-
    (\ell+1/2)f_{j,\ell+1,m}.
\end{align}
The bound \eqref{eq:simple-ineq} implies that for every $m\ge 0$ and every $d\ge 1$ there exists a constant $c_{d,m}$ such that 
\begin{align} \label{eq:fderivativeatone}
    f_{j,\ell,m}(1) 
=
    \frac{z^m}{(n-1)^{j+1/2}}  
    e^{-\frac{z^2}{2}}
\le
    |z|^m e^{-\frac{z^2}{2}}
\le
    \frac{c_{d,m}}{1+|z|^d},\quad n\ge2.
\end{align}
Combining this with \eqref{eq:fderivative} we conclude that for all $d>0$, $p\ge1,\ m\ge 0$ there exists $c_{p,d,m}$
such that
\begin{equation}\label{eq:fderivestim1}
    \left|
    \frac{d^\nu}{d u^\nu}f_{j,\ell,m}(u)\big|_{u=1}
    \right|
\le
    \frac{c_{p,d,m}}{1+|z|^{d}}, \quad \nu\le p,\ n\ge2.
\end{equation}
According to \eqref{eq:fderivative}, every derivative $f_{j,\ell,m}^{(\nu)}$
can be represented as a finite linear combination of functions
$f_{j+j_1,\ell+\ell_1,m+m_1}$ with appropriate nonnegative indices $j_1,\ell_1,m_1$. Applying \eqref{eq:fderivestim1} and \eqref{eq:fderivative2} to every term in the representation for $f_{j,\ell,m}^{(\nu)}$, we infer that an estimate similar to \eqref{eq:fderivative2} holds for every derivative
$f_{j,\ell,m}^{(\nu)}$. Choosing $p$ and $d$ so large that the remainders become $O\left(\frac{1}{(\min(\sqrt{n},1+|z|))^{r+3}}\right)$, we finally obtain
\begin{multline} \label{eq:derivativesumestim}
    \sum_{\nu=1}^{p-1}
        \frac{B_\nu}{\nu!}
        \big(f_{j,\ell,m}^{(\nu-1)}(n-1) - f_{j,\ell,m}^{(\nu-1)}(1)\big) 
\\= 
    e^{-\frac{z^2}{2n}}
    \frac{z^m}{n^{\ell+1/2}}
    \sum_{q=0}^{\lceil \frac{r+m}{2}\rceil-\ell+1}
        \frac{P^{(1)}_{q,j,\ell}(z^2/n)}{n^{q}}
+
    O\left(\frac{1}{(\min(\sqrt{n},1+|z|))^{r+3}}\right).
\end{multline}
Thus, it remains to analyse the integral summand on the right-hand side of \eqref{eq:EM}. Set
\begin{align*}
    M_{j,\ell,m}
:=
    \int_{1}^{n-1} f_{j,\ell,m}(u) du.
\end{align*}
For $j\ge 1$, integrating by parts, we obtain
\begin{multline*}
    (1/2-j)M_{j,\ell,m} 
=
    \int_1^{n-1} 
        \frac{z^m}{(n-u)^{j-1/2}}
        \left(
            \frac{-(\ell+1/2)}{u^{\ell+3/2}}
        +
            \frac{z^2}{2  u^{\ell+5/2}}
        \right)
    e^{-\frac{z^2}{2  u}} du\\
    -\frac{z^m}{(n-u)^{j-1/2}u^{\ell+1/2}}
    e^{-\frac{z^2}{2  u}}\bigg|_{1}^{n-1}
\end{multline*}
or, equivalently,
\begin{align} \label{eq:mrecurent}
   (1/2-j)M_{j,\ell,m} 
=&
    \frac{1}{2}
    M_{j-1,\ell+2,m+2} 
-
    (\ell+1/2)M_{j-1, \ell +1,m}
\\ \nonumber
-&
    \frac{z^m}{(n-1)^{\ell+1/2}} 
    e^{-\frac{z^2}{2(n-1)}} 
+
    \frac{z^m}{(n-1)^{j-1/2}}
    e^{-\frac{z^2}{2}}.
\end{align}
Applying \eqref{eq:fderivative2} to the first remainder term and \eqref{eq:fderivativeatone} to the second one, we obtain 
\begin{multline} \label{eq:mjellm}
    M_{j,\ell,m} 
=
    \sum_{q=0}^{j}
        \gamma_{q,j,\ell}
    M_{0,\ell+j+q,m+2q} 
\\
+
    e^{-\frac{z^2}{2n}} 
    \frac{z^m}{n^{\ell+1/2}}
    \sum_{q=0}^{\lceil \frac{r+m}{2} \rceil-\ell+1}
        \frac{P^{(2)}_{q,j,\ell}(z^2/n)}{n^{q}}
+
    O\left(\frac{1}{(\min(\sqrt{n},1+|z|))^{r+3}}\right)
\end{multline}
with some polynomials $P^{(2)}_{q,j,\ell}$ of degree $q$ and some reals $\gamma_{q,j,\ell}$.

Next, combining equations \eqref{eq:fderivative2}, \eqref{eq:mjellm}, \eqref{eq:eulerintegralrest2} and \eqref{eq:derivativesumestim}, we have
\begin{multline} \label{eq:lem:interimresults}
    \sum_{k=1}^{n-1} \frac{z^m}{k^{j+1/2}(n-k)^{\ell+1/2}} 
    e^{-\frac{z^2}{2 (n-k)}} 
\\
= 
    \sum_{q=0}^{j}
        \gamma_{q,j,\ell}
    \int_1^{n-1}
    \frac{z^{m+2q}}{\sqrt{n-u}}u^{-\ell-j-q-1/2}e^{-\frac{z^2}{2 u}}du
\\
+
    e^{-\frac{z^2}{2n}} 
    \frac{z^m}{n^{\ell+1/2}}
    \sum_{q=0}^{\lceil \frac{r+m}{2} \rceil-\ell+1}
        \frac{P^{(3)}_{q,j,\ell}(z^2/n)}{n^{q}}+
    O\left(\frac{1}{(\min(\sqrt{n},1+|z|))^{r+3}}\right)
\end{multline}
with some polynomials $P^{(3)}_{q,j,\ell}$ of degree $q$ and some reals $\gamma_{q,j,\ell}$.

To deal with the integrals in the previous formula, we shall use the identity
\begin{multline}\label{eq:integralsplitting}
    \int_1^{n-1}
        \frac{z^m}{\sqrt{n-u}}u^{-b-1/2}e^{-\frac{z^2}{2 u}}du 
=
    \int_1^{n-1} f_{0,b,m}(u)du
\\
=
    \int_0^{n} f_{0,b,m}(u)du
-
    \int_0^1 f_{0,b,m}(u)du
-
    \int_{n-1}^n f_{0,b,m}(u)du.
\end{multline}
For every fixed $b$ and $z^2 \ge 2b+1$ the function $u^{-b-1/2}e^{-\frac{z^2}{2 u}}$  is increasing on $(0,1)$. Hence 
$f_{0,b,m}(u)$ is increasing as well, and 
the integral over $[0,1]$ does not exceed the value of $f_{0,b,m}(u)$ at $u=1$. Consequently, for any $d > 0$ it holds that
\begin{align} \label{eq:firstendintegral}
    \int_0^1 f_{0,b,m}(u)du
\le 
    \frac{|z|^m}{\sqrt{n-1}}
    e^{-\frac{z^2}{2}}
     = O\left(|z|^{-d}\right).
\end{align}
For the integral over $[n-1,n]$ we apply the substitution $u \mapsto n-u$
to get
\begin{align*}
    \int_{n-1}^n 
        f_{0,b,m}(u)du
=
    \int_0^1 
        \frac{z^m}{\sqrt{u}(n-u)^{b+1/2}}
        e^{-\frac{z^2}{2 (n-u)}}
    du.
\end{align*}
Similarly to \eqref{eq:another-expan} we have
\begin{multline*}
    z^m
    e^{-\frac{z^2}{2(n-u)}} 
=
    z^m
    e^{-\frac{z^2}{2n}}
\cdot
    e^{-\frac{uz^2}{2n(n-u)}}
\\
=
    e^{-\frac{z^2}{2 n}}
    z^m
    \sum_{\nu=0}^{ d-1 }
        \frac{(-1)^\nu}{2^\nu}
        \frac{u^\nu z^{2\nu}}{n^\nu(n-u)^\nu}
+
    O\left( \frac{1}{n^{d-m/2}} \right).
\end{multline*}
Then
\begin{multline*}
    \int_0^1 
        \frac{z^m}{\sqrt{u}(n-u)^{b+1/2}}
        e^{-\frac{z^2}{2 (n-u)}}
    du
\\
=
    e^{-\frac{z^2}{2 n}}
    z^m
    \sum_{\nu=0}^{ d-1 }\frac{(-1)^\nu}{2^\nu}
        \frac{z^{2\nu}}{n^\nu}
        \int_0^1  \frac{u^{\nu-1/2}}{(n-u)^{b+\nu+1/2}}du
\\
+
    O(n^{-d-b-1/2+m/2}).
\end{multline*}
Expanding $(n-u)^{-b-\nu-1/2}$ into a Taylor series, we conclude that there exists $\theta(u) \in [0,u] \subset [0,1]$ such that 
\begin{multline*}
    \int_0^1  
        \frac{u^{\nu-1/2}}{(n-u)^{b+\nu+1/2}}du
=
    \sum_{\mu=\nu}^{d-1} 
        \frac{C_{b,\nu,\mu}}{n^{b+\mu+1/2}} 
        \int_0^1 u^{\mu-1/2}du 
\\
+
    C_{b,\nu,d}
        \int_0^1 \frac{u^{d-1/2}}{(n-\theta(u))^{b+d+1/2}}du
\\
=
    \sum_{\mu=\nu}^{d-1} 
        \frac{C_{b,\nu,\mu}}{(\mu+1/2)n^{b+\mu+1/2}}
+
    O\left( n^{-d-b-1/2}\right).
\end{multline*}
Choosing now $d=\lceil \frac{r+m}{2} \rceil - b+1$, we arrive at the equality
\begin{multline} \label{eq:secondendintegral}
    \int_0^1 
        \frac{z^m}{\sqrt{u}(n-u)^{b+1/2}}
        e^{-\frac{z^2}{2 (n-u)}}
    du
=
    e^{-\frac{z^2}{2 n}}
    \frac{z^m}{n^{b+1/2}}
    \sum_{\nu=0}^{\lceil \frac{r+m}{2} \rceil-b}
        \frac{P^{(4)}_{\nu,b}(z^2/n)}{n^{\nu}}
\\
+
    O\left(\frac{1}{n^{(r+3)/2}}\right),
\end{multline}
where $P^{(4)}_{\nu,b}$ are again some polynomials of degree $\nu$. 

Substituting $u=t/n$, we obtain
\begin{align*}
    \int_0^{n} f_{0,b,m}(t)dt
&=
    n\int_{0}^1 f_{0,b,m}(un)du 
\\
&=
    n\int_0^1 
    \frac{z^m}{(n-un)^{1/2}}
    {(un)^{-b-1/2}}
    e^{-\frac{z^2}{2  n u}}du 
\\
&=
    \frac{z^m}{n^{b}}
    \int_0^1 
    \frac{1}{\sqrt{1-u}} 
    {u^{-b-1/2}}
    e^{-\frac{z^2}{2un}}du.
\end{align*}
By Lemma \ref{prop:integral}, for every $b\ge 1$ we have
\begin{align*}
    \frac{z^m}{n^{b}}
    \int_0^1 
    \frac{1}{\sqrt{1-u}} 
    {u^{-b-1/2}} 
    e^{-\frac{z^2}{2un}}du 
=
    {\rm sgn}(z)
    \sqrt{2\pi} 
    e^{-\frac{z^2}{2n}}
    \frac{z^m}{n^b}
    \sum_{k=0}^{b-1} 
        \frac{
       (2k-1)!! \binom{b-1}{k}}
       {(z/\sqrt{n})^{2k+1}};
\end{align*}
recall that $(-1)!!=1$.
Hence, using \eqref{eq:firstendintegral} and \eqref{eq:secondendintegral}, we obtain
\begin{multline*} 
    \int_1^{n-1}
    \frac{z^m}{\sqrt{n-u}}u^{-b-1/2}e^{-\frac{z^2}{2 u}}du
\\
=
    {\rm sgn}(z)\sqrt{2\pi} 
    e^{-\frac{z^2}{2n}}
    \frac{z^m}{n^b}
    \sum_{k=0}^{b-1}
        \frac{
       (2k-1)!! \binom{b-1}{k}}
       {(z/\sqrt{n})^{2k+1}}
\\
+
    e^{-\frac{z^2}{2 n}}
    \frac{z^m}{n^{b+1/2}}
    \sum_{\nu=0}^{\lceil \frac{r+m}{2} \rceil -b+1}
        \frac{P^{(4)}_{\nu,b}(z^2/n)}{n^{\nu}}
+
    O\left(\frac{1}{(\min(\sqrt{n},1+|z|))^{r+3}}\right).
\end{multline*}
Applying this representation with $m+2q$ instead of $m$ and with $b=\ell+j+q$ to the integrals on the right-hand side of \eqref{eq:lem:interimresults}, we obtain
\begin{multline}\label{eq:detaileddocomp}
    \sum_{q=0}^{j}
        \gamma_{q,j,\ell}
    \int_1^{n-1}
    \frac{z^{m+2q}}{\sqrt{n-u}}u^{-\ell-j-q-1/2}e^{-\frac{z^2}{2 u}}du
\\
=
    {\rm sgn}(z)e^{-\frac{z^2}{2n}}
    \frac{\sqrt{2\pi} z^m}{n^{\ell+j}}
    \sum_{q=0}^{j}
        \gamma_{q,j,\ell}
        \frac{z^{2q}}{n^q}
    \sum_{k=0}^{\ell+j+q-1} 
        \frac{(2k-1)!! \binom{\ell+j+q-1}{k}}
        {(z/\sqrt{n})^{2k+1}}
\\
+
    e^{-\frac{z^2}{2n}}
    \frac{z^m}{n^{\ell+j}}
    \sum_{q=0}^{j}
        \gamma_{q,j,\ell}
        \frac{z^{2q}}{n^q}
    \sum_{k=0}^{\lceil \frac{r+m}{2} \rceil -\ell-j-q+1}
        \frac{P^{(4)}_{k,\ell+j+q}(z^2/n)}{n^{k+1/2}}
\\
+
    O\left(\frac{1}{(\min(\sqrt{n},1+|z|))^{r+3}}\right).
\end{multline} 
Denoting
\begin{align} \label{eq:defqjellm}
    Q_{j,\ell,m}(t)
=
    t^m
    \sum_{q=0}^{j}
        \gamma_{q,j,\ell} t^{2q}
        \sum_{k=0}^{\ell+j+q-1}
            \frac{(2k-1)!!
            \binom{\ell+j+q-1}{k}}
            {t^{2k+1}}
\end{align}
and introducing polynomials $P^{(5)}_{k,j,\ell}$ of degree $k+j$ 
as follows 
\[
    P^{(5)}_{k,j,\ell}(t) 
:=
    \sum_{q=0}^{j \wedge (\lceil \frac{r+m}{2} \rceil -\ell-j-k+1)}
        \gamma_{q,j,\ell} t^q P^{(4)}_{k,\ell+j+q}(t)
\]
we obtain 
\begin{align*}
    &\sum_{q=0}^{j}
        \gamma_{q,j,\ell}
    \int_1^{n-1}
        \frac{z^{m+2q}}{\sqrt{n-u}}u^{-\ell-j-q-1/2}e^{-\frac{z^2}{2 u}}du
=
    {\rm sgn}(z)\sqrt{2\pi}e^{-\frac{z^2}{2n}}
    \frac{Q_{j,\ell,m}(z/\sqrt{n})}{n^{\ell+j-m/2}}
    \\
&\hspace{1cm}+
    \frac{z^m}{n^{\ell+j+1/2}}e^{-\frac{z^2}{2n}}
    \sum_{k=0}^{
    \lceil \frac{r+m}{2}\rceil - \ell+1}
        \frac{P_{k,j,\ell}^{(5)}(z^2/n)}{n^{k}}
+
    O\left(\frac{1}{(\min(\sqrt{n},1+|z|))^{r+3}}
    \right).
\end{align*}
Combining this with \eqref{eq:lem:interimresults} we obtain the desired equality.
\end{proof}
\begin{lemma} \label{lem:qjellsmall}
If nonnegative integers $j,\ell,m$ satisfy $2\ell + 2j \ge m$ then
\begin{align} \label{eq:qjellsmall}
    e^{-\frac{z^2}{2n}}
    \frac{Q_{j,\ell,m}(z/\sqrt{n})}{n^{\ell+j-m/2}}
=
    O\left(\frac{1}{(\min(\sqrt{n},1+|z|))^{2\ell+2j-m}}
    \right)
\end{align}
for all $z$ with $|z|\ge 1$.
\end{lemma}
\begin{proof}
Denote $\alpha = \ell + j - m/2$.
It is clear from the definition of $Q_{j,\ell,m}$ that the left-hand side of \eqref{eq:qjellsmall} can be represented as a sum of terms, each of which has the form
\begin{align*}
    e^{-\frac{z^2}{2n}}\frac{1}{n^{\alpha}}
    \left(\frac{z}{\sqrt{n}}\right)^{\beta}.
\end{align*}
Applying \eqref{eq:simple-ineq} we get,
for every $\beta \ge 0$,
\begin{align*}
    e^{-\frac{z^2}{2 n}}\frac{1}{n^{\alpha}}
    \left(\frac{z}{\sqrt{n}}\right)^{\beta}
=
    O\left(\frac{1}{n^{\alpha}}\right).
\end{align*}
Furthermore, it follows from the definition of $Q_{j,\ell,m}$  that $\alpha$ and $\beta$ always satisfy $\alpha +\beta/2 \ge 1/2$. Thus, for $\beta<0$ and
$|z|\ge1$ one has
\begin{align*}
    e^{-\frac{z^2}{2 n}}
    \frac{1}{n^{\alpha}}
    \left(\frac{z}{\sqrt{n}}\right)^{\beta}
=
    e^{-\frac{z^2}{2 n}}
    \frac{1}{z^{-\beta}n^{\alpha+\beta/2}}
=
    O\left(\frac{1}{(\min(\sqrt{n},1+|z|))^{2\alpha}}
    \right).
\end{align*}
Thus, the lemma is proved.
\end{proof}
\begin{lemma} \label{lem:gammas}
    The coefficients $\gamma_{q,j,\ell}$ are given by
\begin{align} \label{eq:gammas}
    \gamma_{q,j,\ell}
=
    \frac{(-1)^q2^j}{2^q(2j-1)!!}
    \sum
        \prod_{i=1}^{j-q}
            (\ell+2a_i-i-1/2),
\end{align}
where the sum is taken over all strictly increasing tuples $(a_1, a_2, \dots, a_{j-q})$ chosen from $(1,2,\dots, j)$. In the case $j=q$, the sum is assumed to be equal to $1$.
\end{lemma}
\begin{proof}
To compute $\gamma_{q,j,\ell}$ we consider the recursive equation which follows from \eqref{eq:mrecurent}:
\begin{align*}
    \gamma_{q,j,\ell} 
=
    \frac{\ell+1/2}{j-1/2} \gamma_{q,j-1,\ell+1}
-
    \frac{1}{2(j-1/2)} \gamma_{q-1,j-1,\ell+2}.
\end{align*}
Then, setting
\begin{align*}
    \theta_{q,j,\ell} = \gamma_{q,j,\ell}
    \left(
        \frac{(-1)^q 2^j}{2^q (2j-1)!!}
    \right)^{-1},
\end{align*}
we get the equation
\begin{align} \label{eq:thetarec}
    \theta_{q,j,\ell} 
=
    (\ell+1/2)\theta_{q,j-1,\ell+1} 
+
    \theta_{q-1,j-1,\ell+2}
\end{align}
with initial conditions 
\begin{align*}
    \theta_{0,0,\ell} = 1 
    \quad \quad \text{and} \quad \quad
    \theta_{q,j,\ell} = 0 \quad \text{for} \quad q>j.
\end{align*}
Obviously $\theta_{j,j,\ell}=1$. Next we show that for $q<j$ one has the equality
\begin{align*}
    \theta_{q,j,\ell} 
=
    \sum
        \prod_{i=1}^{j-q}
            (\ell+2a_i-i-1/2),
\end{align*}
where the sum is taken over all strictly increasing tuples $(a_1, a_2, \dots, a_{j-q})$ chosen from $(1,2,\dots, j)$, that is  $a_u<a_v$ for $u<v$. Heuristically, we apply \eqref{eq:thetarec} $j$ times and for each subset $\{a_1, a_2, \dots, a_{j-q}\}$ the value $a_i$ is the index of the step at which we choose multiplication by $\ell+1/2$ for the $i$-th time. 

The bijection of summands in the definition of $\theta_{q,j,\ell}$ corresponding to \eqref{eq:thetarec} can be seen as follows. Every summand in $\theta_{q,j,\ell}$ such that $a_1 > 1$ corresponds to exactly one summand in $\theta_{q-1,j-1,\ell+2}$ by letting $\bar{a}_i = a_i-1$:
\begin{align*}
    \prod_{i=1}^{(j-1)-(q-1)}
        ((\ell+2)+2\bar{a}_i-i-1/2)
&=
    \prod_{i=1}^{j-q}
        ((\ell+2)+2(a_i-1)-i-1/2)
\\
&=
    \prod_{i=1}^{j-q}
        (\ell+2a_i-i-1/2).
\end{align*}
For all other subsets $a_1=1$ and then for $\bar{a}_i = a_{i+1}-1$ one has:

\begin{align*}
    \prod_{i=1}^{(j-1)-q}
        ((\ell+1)+2\bar{a}_i-i-1/2)
&=
    \prod_{i=1}^{j-q-1}
        ((\ell+1)+2(a_{i+1}-1)-i-1/2)\\
&=
    \prod_{i=2}^{j-q}
        ((\ell+1)+2(a_{i}-1)-(i-1)-1/2)\\
&=
    \prod_{i=2}^{j-q}
        (\ell+2a_i-i-1/2).
\end{align*}
Additional multiplication by $(\ell+1/2)$ in the recurrence corresponds to the missing factor for $i=1$:
\begin{align*}
    \ell +2a_1-1-1/2 = \ell + 2\cdot 1-3/2=\ell+1/2.
\end{align*}
Thus, the lemma is proved.
\end{proof}

\begin{lemma} \label{lem:eulerintegralrest}
There exist polynomials $P_{i,j,\ell}$ of degree $i$
such that
\begin{multline*}
    \int_0^1 \big[B_p(u) - B_p\big] 
        \sum_{k = 1}^{n-2} 
            f^{(p)}_{j,\ell,m}(k+1-u)du 
\\
=
    e^{-\frac{z^2}{2n}} 
    \frac{z^m}{n^{\ell+1/2}}
    \sum_{i=0}^{j+p-1}
        \frac{P_{i,j,\ell}(z^2/n)}{n^{i}} 
\\
+
    O\left(\frac{1}{\min (n^{j+\ell+p-m/2}, |z|^{2(j+\ell+p)-m})}\right).
 \end{multline*}
\end{lemma}
\begin{proof}
We start with the case $p=0$. 
Applying Lemmas \ref{lem:firsthalfsum} and \ref{lem:secondhalfsumgeneral} with $t=1-u,a=j+1/2$ and $b=\ell+1/2$, we deduce that, uniformly in $u\in [0,1]$,
\begin{multline*}
   \sum_{k = 1}^{n-2}  f_{j,\ell,m}(k+1-u) 
=
    \frac{z^m}{n^{\ell+1/2}}
    e^{-\frac{z^2}{2 n}}
    \sum_{\nu=0}^{j-1}
        \frac{z^{2\nu}}{n^{\nu}}
        \sum_{i=\nu}^{j-1}
            \frac{c_{i,\nu}g_{a-i}(1-u)}{n^{i}}
\\
+
    O\left(\frac{1}{n^{j+\ell-m/2}}\right)
+
    O\left( \frac{1}{n^{j+1/2}|z|^{2\ell-m-1}} \right),
\end{multline*}
where
\begin{align*}
    g_{\lambda}(t) = \sum_{k=2}^{\infty}\frac{1}{(k-t)^{\lambda}}.
\end{align*}
Letting
\begin{align*}
    G_\lambda = \int_{0}^1 \big[B_p(u) - B_p\big] g_\lambda(1-u) du
\end{align*}
we then have
\begin{multline*}
    \int_0^1 \big[B_p(u) - B_p\big] \sum_{k = 1}^{n-2} f_{j,\ell,m}(k+1-u) du 
\\
=
    \frac{z^m}{n^{\ell+1/2}}
    e^{-\frac{z^2}{2n}}
    \sum_{\nu=0}^{j-1}
        \frac{z^{2\nu}}{n^\nu}  
        \sum_{i=\nu}^{j-1}
            \frac{c_{i,\nu}G_{j+1/2-i}}{n^{i}}
\\
+
    O\left(\frac{1}{n^{j+\ell-m/2}}\right)
+
    O\left( \frac{1}{n^{j+1/2}|z|^{2\ell-m-1}} \right).
\end{multline*}
In other words,
\begin{multline} \label{eq:zeropintegralreminder}
    \int_0^1 \big[B_p(u) - B_p\big] \sum_{k = 1}^{n-2} f_{j,\ell,m}(k+1-u) du 
\\
=
    e^{-\frac{z^2}{2n}} 
    \frac{z^m}{n^{\ell+1/2}}
    \sum_{i=0}^{j-1}
        \frac{\widehat{P}_{i,j,\ell}(z^2/n)}{n^{i}} 
\\
+
    O\left(\frac{1}{n^{j+\ell-m/2}}\right)
+
    O\left( \frac{1}{n^{j+1/2}|z|^{2\ell-m-1}} \right),
\end{multline}
where $\widehat{P}_{i,j,\ell}$ is a polynomial of degree $i$.

If $p\ge1$ then one can use \eqref{eq:fderivative} to get the following representation:
\begin{align*}
    f_{j,\ell,m}^{(p)}(u) = \sum_{0\le q_1+q_2 \le p}\lambda_{q_1,q_2}f_{a, b, m+2q_1}(u),
\end{align*}
where
\begin{align*}
    a = j+p-q_1-q_2, \quad \quad b = \ell+q_2+2q_1.
\end{align*}
For the term $f_{a,b,m+2q_1}$, the decomposition \eqref{eq:zeropintegralreminder} yields the main contribution
\begin{align*}
    \lambda_{q_1, q_2}
    e^{-\frac{z^2}{2n}}
    \frac{z^m}{n^{b+1/2}}
    \sum_{i=0}^{a-1}
        \frac{\widehat{P}_{i,a,b}(z^2/n)}{n^{i}}.
\end{align*}
The remainder can be estimated as
\begin{align*}
    O\left(\frac{1}{n^{j+\ell+p-m/2}}\right)
+
    O\left( \frac{1}{n^{a+1/2}|z|^{2b-m-2q_1-1}} \right).
\end{align*}
Since $a+b-q_1 = j+\ell+p$ the remainder is of order
\begin{align*}
    O\left( 
        \frac{1}{\min 
        (
            n^{j+\ell+p-m/2},
            |z|^{2(j+\ell+p)-m}
        )}
    \right).
\end{align*}
After taking the linear combination, we obtain
\begin{multline*}
    \int_0^1 \big[B_p(u) - B_p\big] 
        \sum_{k = 1}^{n-2} 
            f_{j,\ell,m}^{(p)}(k+1-u) du 
\\
=
    e^{-\frac{z^2}{2n}} 
    \frac{z^m}{n^{\ell+1/2}}
    \sum_{i=0}^{j+p-1}
        \frac{P_{i,j,\ell}(z^2/n)}{n^{i}} 
\\
+
    O\left(\frac{1}{\min (n^{j+\ell+p-m/2}, |z|^{2(j+\ell+p)-m})}\right)
\end{multline*}
with some appropriate polynomials $P_{i,j,\ell}$ of degree $i$.
\end{proof}

\section{Asymptotic expansions for $\pr(S_n=x)$}
In this short section we adapt an asymptotic expansion from Petrov's book \cite{Petrov} for our purposes.
\begin{proposition} \label{prop:locprob}
Under the conditions of Theorem \ref{thm:lattice} one has 
\begin{align*}
    \pr(S_n = x) 
=
    e^{-\frac{z^2}{2 n}}
    \sum_{j=0}^{2r+2}
        \frac{P_j^{(0)}(z)}{n^{j+1/2}}     
+
    w_n(x)
    ,\quad n\ge1,
\end{align*}
where every $P^{(0)}_j(z)$ is a polynomial of degree 
$\lf\frac{3j}{2}\rf$, $P^{(0)}_0(z) = 1/(\sigma \sqrt{2 \pi})$
and $z=x/\sigma$.\\
The remainder term $w_n$ satisfies
\begin{align} \label{eq:wbound}
    |w_n(x)| \le \frac{\varepsilon(n)}{(\sqrt{n}+|x|)^{r+2}},
\end{align}
where $\varepsilon(n)$ is a monotone decreasing to zero function.
\end{proposition}
\begin{proof}
    By Theorem VII.3.16 in \cite{Petrov} with $k = r+3$ we have,
uniformly in $t$,
\begin{align} \label{eq:asymp4}
    (1+|t|)^{r+3} \left(\sigma \sqrt{n} \pr (S_n = x) - \frac{1}{\sqrt{2\pi}} e^{-t^2/2} - \sum_{\nu=1}^{r+1} \frac{q_{\nu}(t)}{n^{\nu/2}} 
    \right)
    = o\left(\frac{1}{n^{(r+1)/2}}\right),
\end{align}
where $t=\frac{x}{\sigma \sqrt{n}}$ and
\begin{align} \label{eq:qpolydef}
    q_{\nu}(t) = \frac{1}{\sqrt{2\pi}} e^{-t^2/2}\sum H_{\nu + 2s}(t) \prod_{m=1}^\nu \frac{1}{k_m!}
\left(\frac{\gamma_{m+2}}{(m+2)!\sigma^{m+2}}\right)^{k_m}.
\end{align}
The sum is taken over all integer non-negative solutions $(k_1, k_2, \dots, k_\nu)$ to the equation $k_1 + 2k_2 + \dots +\nu k_\nu = \nu$
and $s$ denotes the sum of all $k_i$'s, that is, $s = k_1 + k_2 + \dots + k_\nu$.
Here $\gamma_\nu$ is the cumulant of order $\nu$ of the random variable $X_1$, and $H_m(t)$ is the $m$-th Hermite polynomial:
\begin{equation*}
H_m(t) = (-1)^m e^{t^2/2}\frac{d^m}{d t^m}e^{-t^2/2}.
\end{equation*}
Every $H_m(t)$ has degree $m$ and contains $t^j$ if and only if $j$ and $m$ are of the same parity. 
Combining this with the  definition of $q_\nu(t)$ we infer that the polynomial $\widehat{q}_\nu(t):= q_\nu(t)e^{t^2/2}$ contains $t^j$ iff $j$ and $\nu$ are of the same parity. Furthermore, the degree of $\widehat{q}_\nu(t)$ is $3\nu$. In other words, there exist reals $A_{\nu,j}$ such that 
\begin{align} \label{eq:qpolycoeff}
    \widehat{q}_\nu(t)=
\sum_{j=0}^{\lf\frac{3\nu}{2}\rf}A_{\nu,j}t^{3\nu-2j}.
\end{align}
Therefore, for $t=\frac{x}{\sigma\sqrt{n}} = \frac{z}{\sqrt{n}}$ one has
\begin{align*}
    \frac{1}{\sigma\sqrt{n}}
    \sum_{\nu=1}^{r+1} 
        \frac{\hat{q}_{\nu}(\frac{z}{\sqrt{n}})}{n^{\nu/2}} 
&= \frac{1}{\sigma\sqrt{n}}
    \sum_{\nu=1}^{r+1}
    n^{-\nu/2}\sum_{j=0}^{\lf\frac{3\nu}{2}\rf}A_{\nu,j}\left(\frac{z}{\sqrt{n}}\right)^{3\nu-2j}\\
&=
    \frac{1}{\sigma}
    \sum_{\nu=1}^{r+1}
    \sum_{j=0}^{\lf\frac{3\nu}{2}\rf} 
        A_{\nu,j}
        \frac{z^{3\nu-2j}}{n^{2\nu-j+1/2}}\\
&=
    \frac{1}{\sigma}
    \sum_{m=1}^{2r+2}\frac{1}{n^{m+1/2}}
    \sum_{\nu=\max \{1, \lceil\frac{m}{2}\rceil\}}^{\min\{2m,r+1\}}
        A_{\nu,2\nu-m}z^{2m-\nu}.
\end{align*}
Thus, 
\begin{align*}
    \frac{1}{\sigma\sqrt{n}}
    \sum_{\nu=1}^{r+1} 
        \frac{\hat{q}_{\nu}(x/(\sigma \sqrt{n}))}{n^{\nu/2}} 
=
    \sum_{j=1}^{2r+2} \frac{P^{(0)}_j(z)}{n^{j+1/2}},
\end{align*}
where
\begin{align} \label{eq:aqjviaA}
    P_{j}^{(0)}(z) 
=
    \frac{1}{\sigma} 
    \sum_{\mu=\max\{0,2j-r-1\}}
    ^{\lf \frac{3j}{2}\rf}
        A_{2j-\mu,3j-2\mu} z^\mu.
\end{align}

Set $P^{(0)}_0(z) = 1/(\sigma \sqrt{2 \pi})$. Then dividing \eqref{eq:asymp4} by $\sigma \sqrt{n}$ one gets
\begin{align*}
    \left(
        1+\frac{|x|}{\sigma \sqrt{n}}\right)^{r+3}
    \left(
        \pr (S_n = x) 
    -
        e^{-\frac{z^2}{2n}} 
        \sum_{j=0}^{2r+2} 
            \frac{P^{(0)}_j(z)}{n^{j+1/2}} 
    \right)
    = o\left(\frac{1}{n^{(r+2)/2}}\right).
\end{align*}
In other words,
\begin{align*}
|w_n(x)|=    \left| 
        \pr (S_n = x) 
    -
        e^{-\frac{z^2}{2 n}} 
        \sum_{j=0}^{2r+2}
            \frac{P_j^{(0)}(z)}{n^{j+1/2}} 
    \right|
\le
    \frac{\varepsilon(n) \sqrt{n}}{(\sqrt{n} + |x|)^{r+3}}
\le 
    \frac{\varepsilon(n)}{(\sqrt{n}+|x|)^{r+2}}.
\end{align*}
Thus, the proof of the proposition is finished.
\end{proof}
\section{Asymptotic expansion for $\pr(S_n=x,\tau<n)$}
\label{sec:4}
This section contains the main step in the proof of Theorem~\ref{thm:lattice}. More precisely, we derive an expansion for $\pr(S_n=x,\tau<n)$, which is valid for all $x\in\mathbb{Z}$. The fact that we do not restrict our attention to positive values of $x$ allows us later to show that some elements of the expansion have zero coefficients.

Define 
\begin{align*}
    \mathcal{P}_n(x) 
:=
    \sum_{k=1}^{n-1}
        \sum_{y\le|x|/4}
        \pr(S_k = -y, \tau = k) 
        \left(
            \pr(S_{n-k} = x+y) - w_{n-k}\left(x+y\right)
        \right),
\end{align*}
where $w_n(x)$ is defined in Proposition~\ref{prop:locprob}.
Substituting this into \eqref{eq:tau-decomp}, we obtain
\begin{align} \label{eq:maindecomp}
\nonumber
\pr (S_n = x, \tau < n) 
&=\mathcal{P}_n(x) 
+
    \sum_{y\le|x|/4}
    \sum_{k=1}^{n-1}
        \pr(S_k = -y, \tau = k) 
        w_{n-k}\left(x+y\right)\\
&+\sum_{k=1}^{n-1}\sum_{y>|x|/4}
\pr(S_k=-y,\tau=k)\pr(S_{n-k}=x+y).
\end{align}
By Proposition~\ref{prop:locprob}, uniformly in $|y|\le |x|/4$ and
$k\le n/2$, 
\begin{align*}
    |w_{n-k}(x+y)| 
    \le 
    \frac{2^{r/2+1}\varepsilon(n/2)}{(\sqrt{n}+|x|)^{r+2}}.
\end{align*}
Therefore,
\begin{align*}
&\left|
    \sum_{y\le|x|/4}
    \sum_{k=1}^{\lfloor \frac{n}{2}\rfloor}
        \pr(S_k = -y, \tau = k) 
        w_{n-k}\left(x+y\right)
        \right| \\
    &\hspace{1cm}
    \le \frac{2^{r/2+1}\varepsilon(n/2)}{(\sqrt{n}+|x|)^{r+2}}
    \sum_{y\le|x|/4}
    \sum_{k=1}^{\lfloor \frac{n}{2}\rfloor}
        \pr(S_k = -y, \tau = k) \le \frac{2^{r/2+1}\varepsilon(n/2)}{(\sqrt{n}+|x|)^{r+2}}.
\end{align*}
Using once again Proposition~\ref{prop:locprob}, we obtain 
\begin{align*}
 &\left|
    \sum_{y\le|x|/4}
    \sum_{k=\lfloor \frac{n}{2}\rfloor+1}^{n-1}
        \pr(S_k = -y, \tau = k) 
        w_{n-k}\left(x+y\right)
        \right| \\   
  &\hspace{1cm}  \le\sum_{k=\lfloor \frac{n}{2}\rfloor+1}^{n-1}
        \pr(\tau = k)
        \frac{2^{r/2+1}\varepsilon(n-k)}{(\sqrt{n-k}+|x|)^{r+2}}. 
\end{align*}
By Theorem 8 from \cite{VW09}, $\pr(\tau = k)\le Ck^{-3/2}$ for all $k$.
Applying this bound and recalling that the function $\varepsilon(k)$ is bounded, we obtain
\begin{align*}
 &\left|
    \sum_{y\le|x|/4}
    \sum_{k=\lfloor \frac{n}{2}\rfloor+1}^{n-1}
        \pr(S_k = -y, \tau = k) 
        w_{n-k}\left(x+y\right)
        \right| \\   
  &\hspace{1cm}  \le Cn^{-3/2} \sum_{j=1}^{n}
        \frac{1}{(\sqrt{j}+|x|)^{r+2}} 
        \le \frac{C}{n^{3/2}(1+|x|)^r}.
\end{align*}
Combining these estimates we conclude that  
\begin{align} \label{eq:R0estim}
\nonumber
    &\left|
    \sum_{y\le|x|/4}
    \sum_{k=1}^{n-1}
        \pr(S_k = -y, \tau = k) 
        w_{n-k}\left(x+y\right)
        \right|\\
&\hspace{2cm}\le \frac{2^{r/2+1}\varepsilon(n/2)}{(\sqrt{n}+|x|)^{r+2}}
  +
  \frac{C}{n^{3/2}(1+|x|)^r}.
\end{align}
By the local limit theorem for strongly aperiodic lattice random walks,
$$
\sup_{x\in\mathbb{Z}}\left|
\sqrt{n}\pr(S_n=x)-\frac{1}{\sqrt{2\pi\sigma^2}}e^{-\frac{x^2}{2\sigma^2n}}
\right|\to0\quad\text{as }n\to\infty.
$$
Therefore, there exists a constant $C$ such that, uniformly in $x$ and $y$,
\begin{equation}
\label{eq:LLT}
\pr(S_{n-k}=x+y)\le\frac{C}{\sqrt{n-k}}.
\end{equation}
Using this bound we can estimate the sum over $y > |x|/4$ in \eqref{eq:maindecomp}:
\begin{align} \label{eq:largeyestim}
\nonumber
    &\sum_{k=1}^{n-1}
    \sum_{y>|x|/4}
        \pr(S_k=-y,\tau=k)\pr(S_{n-k}=x+y)\\
&\hspace{3cm}\le
    C\sum_{k=1}^{n-1}
    \frac{\pr(S_k < -|x|/4, \tau = k)}{\sqrt{n-k}}.
\end{align}
We next notice that
\begin{align*}
&\sum_{k=1}^{\lfloor\frac{n}{2}\rfloor}
    \frac{\pr(S_k < -|x|/4, \tau = k)}{\sqrt{n-k}}\\
    &\hspace{1cm}\le \frac{\sqrt{2}}{\sqrt{n}}
         \sum_{k=1}^\infty\pr(S_k < -|x|/4, \tau = k)
         =\frac{\sqrt{2}}{\sqrt{n}}\pr(-S_\tau>|x|/4).
\end{align*}
Our assumption $\E|X_1|^{r+3}<\infty$ implies that $\E(-S_\tau)^{r+2}$
is finite. Applying now the Markov inequality, we conclude that
\begin{align}
\label{eq:new-bound.1}
\sum_{k=1}^{\lfloor\frac{n}{2}\rfloor}
    \frac{\pr(S_k < -|x|/4, \tau = k)}{\sqrt{n-k}}
    \le \frac{C}{\sqrt{n}(1+|x|)^{r+2}}.
\end{align}
To estimate the sum over $k\ge n/2$, we notice that
\begin{align*}
\pr(S_k<-|x|/4,\tau=k)
=\sum_{u=1}^\infty \pr(S_{k-1}=u,\tau>k-1)\pr(X_1<-u-|x|/4).
\end{align*}
According to Lemma 20 in \cite{VW09},
\begin{align}
\label{eq:conc.f}
    \pr(S_{k-1} = u, \tau > k-1)
\le
    C\frac{u}{k^{3/2}},\quad k,u\ge1.
\end{align}
Therefore, uniformly over $k\ge n/2$,
\begin{align*}
\pr(S_k<-|x|/4,\tau=k)
&\le\frac{C}{n^{3/2}}\sum_{u=1}^\infty u\pr(X_1<-u-|x|/4)\\
&\le \frac{C}{n^{3/2}}\E[X_1^2;X_1<-|x|/4].
\end{align*}
Applying now the Markov inequality, we conclude that 
\begin{align}
\label{eq:new-bound.2}
\nonumber
 \sum_{k=\lfloor\frac{n}{2}\rfloor+1}^{n-1}
    \frac{\pr(S_k < -|x|/4, \tau = k)}{\sqrt{n-k}}
 &\le 
   \frac{C_1}{n^{3/2}(1+|x|)^{r+1}}\sum_{k=1}^{n-1}\frac{1}{\sqrt{n-k}}\\
 &\le \frac{C_2}{n(1+|x|)^{r+1}}.
\end{align}
Combining \eqref{eq:new-bound.1} and \eqref{eq:new-bound.2}, we conclude that  
\begin{align*}
    \sum_{k=1}^{n-1}
    \sum_{y>|x|/4}
        \pr(S_k=-y,\tau=k)\pr(S_{n-k}=x+y)
\le
    \frac{C}{(\min\{\sqrt{n},(1+|x|)\})^{r+3}}.
\end{align*}
Applying this and \eqref{eq:R0estim} to the corresponding terms in \eqref{eq:maindecomp}, we obtain
\begin{align}\label{eq:maindecomp0.5}
\left|
    \pr (S_n = x, \tau < n) 
-
    \mathcal{P}_n(x)
\right|
\le
    \frac{C}{(\min\{\sqrt{n},(1+|x|)\})^{r+3}}+
    \frac{2^{r/2+1}\varepsilon(n/2)}{(\sqrt{n}+|x|)^{r+2}}.
\end{align}
By Proposition \ref{prop:locprob}, 
\begin{align*}
    \pr(S_{n-k} = x+y) 
-
    w_{n-k}(x+y)
=
    e^{-\frac{(z+\y)^2}{2(n-k)}}
    \sum_{j=0}^{2r+2}
        \frac{P^{(0)}_j(z+\y)}{(n-k)^{j+1/2}},
\end{align*}
where $\y = y/\sigma$ and $z=x/\sigma$.
Consequently,
\begin{align*}
    \mathcal{P}_n(x)
= 
    \sum_{j=0}^{2r+2}
        \sum_{k=1}^{n-1}
            \sum_{y\le|x|/4}
                P^{(0)}_j(z+\y)
                e^{-\frac{(z+\y)^2}{2(n-k)}}
                \frac{\pr(S_k=-y, \tau = k)}{(n-k)^{j+1/2}}.
\end{align*}
Let $a_{q,j}$ denote the coefficients of $P_j^{(0)}$, i.e.,
\begin{align*}
    P^{(0)}_j(z+\y) 
=
    \sum_{q=0}^{\lf \frac{3j}{2}\rf}
        a_{q,j} (z+\y)^q
=
    \sum_{q=0}^{\lf \frac{3j}{2}\rf}
        \sum_{s=0}^q
            a_{q,j} 
            \binom{q}{s}
            z^{q-s}\y^s.
\end{align*}
Letting
\begin{align*}
    T_{j,q,s}(z)
= 
    z^{q-s}
    \sum_{k=1}^{n-1}
        \sum_{y\le\sigma|z|/4}
            \y^s
            e^{-\frac{(z+\y)^2}{2(n-k)}}
            \frac{\pr(S_k=-y, \tau = k)}{(n-k)^{j+1/2}},
\end{align*}
one has
\begin{align} \label{eq:Pndecomp}
    \mathcal{P}_n(x)
= 
    \sum_{j=0}^{2r+2}
        \sum_{q=0}^{\lf \frac{3j}{2}\rf}
            \sum_{s=0}^q
                a_{q,j} 
                \binom{q}{s}
                T_{j,q,s}(z).
\end{align}
\begin{lemma}\label{lem:large_s}
For $s \ge r+2$, the following bound holds:
\begin{align*}
    |T_{j,q,s}(z)| 
\le 
    \frac{C}{\min(\sqrt{n},1+|z|)^{r+3}}.
\end{align*}
\end{lemma}
\begin{proof}
For $s\ge r+2$, we use the fact that $y\le \sigma|z|/4$ to obtain
\begin{align}\label{eq:T1}
\nonumber
    |T_{j,q,s}(z)|
&\le
    C_1 |z|^{q-s}
    \sum_{k=1}^{n-1}
        \frac{
        e^{-\frac{z^2}{8(n-k)}}
        }
        {(n-k)^{j+1/2}}
        \sum_{y\le \sigma |z|/4}
            \hat{y}^s
            \pr(S_k=-y,\tau=k)\\
&\le
    C_2 |z|^{q-r-2}
    \sum_{k=1}^{n-1}
        \frac{
        e^{-\frac{z^2}{8(n-k)}}
        }
        {(n-k)^{j+1/2}}
        \sum_{y\le \sigma |z|/4}
            \hat{y}^{r+2}
            \pr(S_k=-y,\tau=k).
\end{align}
We next notice that 
\begin{align*}
  \sum_{k=1}^{\lfloor\frac{n}{2}\rfloor}
        \frac{
        e^{-\frac{z^2}{8(n-k)}}
        }
        {(n-k)^{j+1/2}}
        \sum_{y\le \sigma |z|/4}
            \hat{y}^{r+2}
            \pr(S_k=-y,\tau=k)  
    \le 2^{j+1/2}\frac{e^{-\frac{z^2}{16n}}}{n^{j+1/2}}
        \E\left[\left(\frac{-S_\tau}{\sigma}\right)^{r+2}\right].
\end{align*}
Consequently, for all $q\ge r+2$,
\begin{align}
\label{eq:T2} 
\nonumber
&|z|^{q-r-2}\sum_{k=1}^{\lfloor\frac{n}{2}\rfloor}
        \frac{
        e^{-\frac{z^2}{8(n-k)}}
        }
        {(n-k)^{j+1/2}}
        \sum_{y\le \sigma |z|/4}
            \hat{y}^{r+2}
            \pr(S_k=-y,\tau=k) \\
&\hspace{1cm}
\le C_1 \frac{(1+|z|)^q}{(1+|z|)^{r+2}}
\frac{e^{-\frac{z^2}{16n}}}{n^{j+1/2}}
\le \frac{C_2}{n^{1/2}(1+|z|)^{r+2}},
\end{align}
in the last step we have used \eqref{eq:simple-ineq} and the fact that $2j-q\ge0$.

To estimate the sum over $k\ge n/2$ we derive first an estimate for $\pr(S_k=-y,\tau=k)$. Combining \eqref{eq:conc.f} with the total probability law, we obtain
\begin{align}
\label{eq:y-tau}
\nonumber
    \pr(S_k=-y,\tau=k)
&=
    \sum_{u=1}^\infty
        \pr(S_{k-1}=u,\tau>k-1)\pr(X_1=-(u+y))\\
&\le
    \frac{C}{k^{3/2}}
    \sum_{u =1}^\infty
        u\pr(X_1=-(u+y))
\le 
    \frac{C}{k^{3/2}}\e[-X_1;-X_1>y].
\end{align}
Then one has
\begin{align} \label{eq:boundr+2moment}
    \sum_{y\le \sigma |z|/4}
        \hat{y}^{r+2}
        \pr(S_k=-y,\tau=k)
&\le \nonumber
    \frac{\sigma |z|}{4}\sum_{y=0}^\infty
        \hat{y}^{r+1}
        \pr(S_k=-y,\tau=k)\\
&\le \nonumber
    \frac{C_1|z|}{k^{3/2}}
    \sum_{y=0}^\infty
        \sum_{u =1}^\infty
            \hat{y}^{r+1}
            u\pr(X_1=-(u+y))\\
&\le
    \frac{C_2|z|}{k^{3/2}} \e |X_1|^{r+3}.
\end{align}
Therefore,
\begin{align*}
&|z|^{q-r-2}\sum_{k=\lfloor\frac{n}{2}\rfloor+1}^{n-1}
        \frac{
        e^{-\frac{z^2}{8(n-k)}}
        }
        {(n-k)^{j+1/2}}
        \sum_{y\le \sigma |z|/4}
            \hat{y}^{r+2}
            \pr(S_k=-y,\tau=k)  \\
&\hspace{2cm}\le 
    C|z|^{q-r-1}n^{-3/2}
        \sum_{k=\lfloor\frac{n}{2}\rfloor+1}^{n-1}
        \frac{
        e^{-\frac{z^2}{8(n-k)}}
        }
        {(n-k)^{j+1/2}}.
\end{align*}

Applying Lemma~\ref{lem:firsthalfsum} with $a=t=0$ and $b=j+1/2$
and recalling that $2j\ge q\ge r+2$, we obtain
\begin{align*}
    \sum_{k=\lf n/2\rf}^{n-1}
        \frac{|z|^{q-r-1}
        e^{-\frac{z^2}{8(n-k)}}
        }
        {(n-k)^{j+1/2}}
\le
    \frac{C}{(1+|z|)^{r}}.
\end{align*}
Consequently,
\begin{align}
\label{eq:T3}
\nonumber
&|z|^{q-r-2}\sum_{k=\lfloor\frac{n}{2}\rfloor+1}^{n-1}
        \frac{
        e^{-\frac{z^2}{8(n-k)}}
        }
        {(n-k)^{j+1/2}}
        \sum_{y\le \sigma |z|/4}
            \hat{y}^{r+2}
            \pr(S_k=-y,\tau=k)\\
&\hspace{3cm}\le \frac{C}{n^{3/2}(1+|z|)^{r}}.
\end{align}
Combining \eqref{eq:T1} --- \eqref{eq:T3}, we conclude that, for $s\ge r+2$,
\begin{align*}
    |T_{j,q,s}(z)| 
\le 
    \frac{C}{\min(\sqrt{n},1+|z|)^{r+3}}.
\end{align*}
Thus the lemma is proved.
\end{proof}
Applying Lemma~\ref{lem:large_s} to \eqref{eq:Pndecomp}, and recalling $z= x/\sigma$, we obtain
\begin{align} \label{eq:pxydecomp1}
    \mathcal{P}_n(x)
= 
    \sum_{j=0}^{2r+2}
        \sum_{q=0}^{\lf \frac{3j}{2}\rf}
            \sum_{s=0}^{q\wedge (r+1)}
                a_{q,j} 
                \binom{q}{s}
                T_{j,q,s}(z)
+
    O\left(\frac{1}{\min(\sqrt{n},1+|z|)^{r+3}}\right).
\end{align}

According to the Taylor formula, there exists 
$\theta:=\theta(z,\hat{y},n,k,s)\in(0,1)$ such that 
\begin{align*}
    e^{-\frac{z\y}{n-k}}
=
    \sum_{\nu=0}^{r-s+1}\frac{(-1)^\nu}{\nu!}\frac{z^\nu \y^\nu}{(n-k)^\nu} 
+ 
    \frac{(-1)^{r-s+2}}{(r-s+2)!}
    e^{-\theta\frac{z\y}{n-k}} 
    \frac{z^{r-s+2}\y^{r-s+2}}{(n-k)^{r-s+2}}.
\end{align*}
Hence
\begin{multline} \label{eq:decomp_withRqs}
    T_{j,q,s}(z) 
= 
    \sum_{\nu=0}^{r-s+1}\frac{(-1)^\nu}{\nu!}
    z^{q-s+\nu}
    \sum_{k=1}^{n-1}
        e^{-\frac{z^2}{2(n-k)}}
        \sum_{y\le\sigma|z|/4}
            \y^{s+\nu}
            e^{-\frac{\y^2}{2(n-k)}}
            \frac{\pr(S_k=-y, \tau = k)}{(n-k)^{j+\nu+1/2}}
\\
+\frac{(-1)^{r-s+2}}{(r-s+2)!}
   R_{q,s}(z),
\end{multline}
where
\begin{align*}
    R_{q,s}(z) 
=
    z^{q+r-2s+2}
    \sum_{k=1}^{n-1}
        e^{-\frac{z^2}{2(n-k)}}
        \sum_{y\le\sigma|z|/4}
            \y^{r+2}
            e^{-\theta\frac{z\y}{n-k}}
            e^{-\frac{\y^2}{2(n-k)}}
            \frac{\pr(S_k=-y, \tau = k)}{(n-k)^{j+r-s+5/2}}.
\end{align*}
\begin{lemma}\label{lem:Rqs_bound}
The following bound holds
\begin{align*}
    |R_{q,s}(z)|
\le
    \frac{C}{(\min\{\sqrt{n},1+|z|\})^{r+3}}.
\end{align*}
\end{lemma}
\begin{proof}
It is easy to see that 
$$
    e^{-\frac{z^2}{2 (n-k)}}
    e^{-\theta\frac{z\y}{n-k}}
    e^{-\frac{\y^2}{2 (n-k)}}
\le
    e^{-\frac{z^2}{4(n-k)}}
$$
for all $z$, all $\y\le|z|/4$ and all $\theta\in(0,1)$. Therefore,
\begin{align*}
    |R_{q,s}(z)| 
\le
    |z|^{q+r-2s+2}
    \sum_{k=1}^{n-1}
        e^{-\frac{z^2}{4(n-k)}}
        \frac{
        \sum\limits_{y\le \sigma|z|/4}
        \y^{r+2}\pr(S_k=-y, \tau = k)
        }{(n-k)^{j+r-s+5/2}}.
\end{align*}
For the sum over $k\le n/2$ we have
\begin{align*}
    \sum_{k=1}^{\lf \frac{n}{2}\rf}
        e^{-\frac{z^2}{4(n-k)}}
        \frac{
        \sum\limits_{y\le \sigma|z|/4}
        \y^{r+2}\pr(S_k=-y, \tau = k)
        }{(n-k)^{j+r-s+5/2}}
\le
    \frac{2^{j+r-s+5/2}e^{-\frac{z^2}{8n}}}{\sigma^{r+2}n^{j+r-s+5/2}}\E[\left(-S_{\tau}\right)^{r+2}].
\end{align*}
Since $2j\ge q \ge s$ and $s \le r+2$ we can apply \eqref{eq:simple-ineq} with $a = q+r-2s+2$ and get
\begin{align*}
    |z|^{q+r-2s+2}\frac{e^{-\frac{z^2}{8n}}}{n^{j+r-s+5/2}}
\le
    \frac{C}{n^{(2j-q+r+3)/2}}
\le
    \frac{C}{n^{(r+3)/2}}.
\end{align*}
Hence
\begin{align}\label{eq:Rqsestim}
    |z|^{q+r-2s+2}
    \sum_{k=1}^{\lf \frac{n}{2}\rf}
        e^{-\frac{z^2}{4(n-k)}}
        \frac{
        \sum\limits_{y\le \sigma|z|/4}
        \y^{r+2}\pr(S_k=-y, \tau = k)
        }{(n-k)^{j+r-s+5/2}}
\le
    \frac{C}{n^{(r+3)/2}}.
\end{align}
For the sum over $k > n/2$ we use the bound \eqref{eq:boundr+2moment} and Lemma \ref{lem:firsthalfsum} to obtain
\begin{align*}
    \sum_{k=\lf \frac{n}{2}\rf+1}^{n-1}
        e^{-\frac{z^2}{4(n-k)}}
        \frac{
        \sum\limits_{y\le \sigma|z|/4}
        \y^{r+2}\pr(S_k=-y, \tau = k)
        }{(n-k)^{j+r-s+5/2}}
\le&
    \sum_{k=\lf \frac{n}{2}\rf+1}^{n-1}
        \frac{C|z|
        e^{-\frac{z^2}{4(n-k)}}
        }{k^{3/2}(n-k)^{j+r-s+5/2}}\\
\le&
    \frac{C|z|}{n^{3/2}(1+|z|)^{2j+2r-2s+2}}.
\end{align*}
Multiplying both sides by $|z|^{q+r-2s+2}$ and combining the resulting inequality with \eqref{eq:Rqsestim}, we obtain
\begin{align*}
    |R_{q,s}(z)|
\le
    \frac{C}{n^{(r+3)/2}}
+
    \frac{C}{n^{3/2}|z|^{r}}
\le
    \frac{C}{(\min\{\sqrt{n},1+|z|\})^{r+3}}.
\end{align*}
Thus the lemma is proved.
\end{proof}

Applying Lemma~\ref{lem:Rqs_bound} to \eqref{eq:decomp_withRqs} we obtain
\begin{multline*}
    T_{j,q,s}(z) 
= 
    \sum_{\nu=0}^{r-s+1}
    \frac{(-1)^\nu}{\nu!}
    z^{q-s+\nu}
    \sum_{k=1}^{n-1}
        e^{-\frac{z^2}{2 (n-k)}}
        \sum_{y\le\sigma|z|/4}
            \y^{s+\nu}
            e^{-\frac{\y^2}{2(n-k)}}
            \frac{\pr(S_k=-y, \tau = k)}{(n-k)^{j+\nu+1/2}}
\\
+
    O\left(\frac{1}{(\min(\sqrt{n},1+|z|))^{r+3}}\right).
\end{multline*}
Applying \eqref{eq:exp-taylor} with $R=r-s-\nu+2$, we have 
$$
    e^{-\frac{\y^2}{2{(n-k)}}}
=
    \sum_{\mu=0}^{\lf\frac{r-s-\nu+1}{2}\rf}
        \frac{h_{2\mu}(0)}{(2\mu)!}
        \frac{\y^{2\mu}}{(n-k)^\mu}
+
    O\left(\frac{\y^{r-s-\nu+2}}{(n-k)^{(r-s-\nu+2)/2}}\right).
$$
Therefore,
\begin{align*}
    &T_{j,q,s}(z)
=
    \sum_{k=1}^{n-1}
        e^{-\frac{z^2}{2 (n-k)}}
        \sum_{\nu+2\mu\le r-s+1}
                \frac{(-1)^\nu z^{q-s+\nu}h_{2\mu}(0)}{\nu!(2\mu)!
                (n-k)^{j+\nu+\mu+1/2}}
                \Theta_k^{(s+\nu+2\mu)}(z)
\\
&\hspace{1cm}
+
    O\left(
        \sum_{k=1}^{n-1}e^{-\frac{z^2}{2(n-k)}}
        \sum_{\nu=0}^{r-s+1}
            |z|^{q-s+\nu}
            \frac{\Theta_k^{(r+2)}(z)}
            {(n-k)^{(r+2j+\nu-s+3)/2}}
    \right)
\\
&\hspace{1cm}
+
    O\left(\frac{1}{(\min(\sqrt{n},1+|z|))^{r+3}}\right),
\end{align*}
where
$$
    \Theta_k^{(h)}(z)
:=
    \sum_{y\le\sigma|z|/4}
    \left(\frac{y}{\sigma}\right)^h
    \pr(S_k=-y,\tau=k).
$$
According to \eqref{eq:boundr+2moment},
$$
\Theta_k^{(r+2)}(z)\le C|z|n^{-3/2}\quad \text{for all }k\ge n/2.
$$
Furthermore, 
$$
\sum_{k=1}^\infty \Theta_k^{(r+2)}(z)
\le \sigma^{-r-2}\E[(-S_\tau)^{r+2}]<\infty.
$$
Using these estimates, we obtain
\begin{align*}
&
    \sum_{k=1}^{n-1}
        e^{-\frac{z^2}{2 (n-k)}}
        |z|^{q-s+\nu}
        \frac{\Theta_k^{(r+2)}(z)}
        {(n-k)^{(r+2j+\nu-s+3)/2}}
\\
&\hspace{1cm}
\le
    C\left(
    \sum_{k=1}^{\lf\frac{n}{2}\rf}
        \frac{|z|^{q-s+\nu}
        e^{-\frac{z^2}{4n}}
        \Theta_k^{(r+2)}(z)}
        {n^{(r+2j+\nu -s +3)/2}}
      +  n^{-3/2}\sum_{k=\lf\frac{n}{2}\rf+1}^{n-1}
        \frac{|z|^{q-s+\nu+1}}{(n-k)^{(r+2j+\nu-s+3)/2}}
        \right)
\\
&\hspace{1cm}
\le
    C_2\left(
    \frac{|z|^{q-s+\nu}
    e^{-\frac{z^2}{4n}}}   
    {n^{(r+2j+\nu-s+3)/2}}
+
    \sum_{k=\lf\frac{n}{2}\rf+1}^{n-1}
        \frac
        {
            |z|^{q-s+\nu}e^{-\frac{z^2}{2 (n-k)}}
        }
        {
            k^{3/2}(n-k)^{(r+2j+\nu -s +3)/2}
        }\right).
\end{align*}
Applying Lemma~\ref{lem:firsthalfsum} to the second summand and \eqref{eq:simple-ineq} to the first one,
we have 
\begin{align*}
    \sum_{k=1}^{n-1}
        e^{-\frac{z^2}{2 (n-k)}}
        |z|^{q-s+\nu}
        \frac{\Theta_k^{(r+2)}(z)}{(n-k)^{(r+2j+\nu-s+3)/2}}
\le
    \frac{C}{\min(\sqrt{n},1+|z|)^{r+3}}.
\end{align*}
Consequently,
\begin{multline} \label{eq:Tjqs}
    T_{j,q,s}(z)
=
    \sum_{k=1}^{n-1}
    e^{-\frac{z^2}{2(n-k)}}
    \sum_{\nu+2\mu\le r-s+1}
            \frac{(-1)^\nu z^{q-s+\nu}h_{2\mu}(0)}{\nu!(2\mu)!(n-k)^{j+\nu+\mu+1/2}}
            \Theta_k^{(s+\nu+2\mu)}(z)
\\
+
    O\left(\frac{1}{(\min(\sqrt{n},1+|z|))^{r+3}}\right).
\end{multline}
We next derive expansions for $\Theta_k^{(h)}(z)$ with $h\le r+1$.

\begin{lemma}\label{lem:sumovery}
    Under the conditions of Theorem \ref{thm:lattice} there exist numbers $b_\ell^{(h)}$ such that, for every $h=0,1,\dots, r+1$ one has
\begin{align*}
   \Theta_{k}^{(h)}
=
    \sum_{\ell=0}^{\lf \frac{r-h}{2}\rf}
        \frac{b_\ell^{(h)}}{k^{\ell+3/2}} + v^{(h)}_k,
\quad k\ge 1,
\end{align*}
where
\begin{align*}
|v_1^{(h)}| \le C, \quad
    |v_k^{(h)}|
\le
    \frac{C\log^{\lfloor \frac{r-h+1}{2} \rfloor} k}{k^{(r-h+4)/2}},
    \quad k\ge 2.
\end{align*}
Moreover,
$$
    \Theta_k^{(h)}(z)
:=
    \sum_{0\le y\le\sigma|z|/4}
        \left(\frac{y}{\sigma}\right)^{h}
        \pr(S_k=-y,\tau = k)
=
    \Theta_k^{(h)}
+
    O\left(\frac{1}{k^{3/2}(1+|z|)^{r-h+2}}\right).
$$
\end{lemma}

\begin{proof}
For $k\ge 2$, by the total probability law,
\begin{align} \label{eq:jumptonegative}
    \pr (S_k = -y, \tau =k) 
=
    \sum_{u=1}^\infty \pr(S_{k-1} 
=
    u, \tau > k-1)\pr(X_1=-(y+u)).
\end{align}
For $h\le r$, applying \eqref{eq:locprobtau} with parameter $r-h$ instead of $r$, we have
\begin{align} \label{eq:localtaudecomp}
    \pr(S_{k-1} = u, \tau > k-1)
= &
    \sum_{\ell= 1}^{\lf \frac{r-h}{2}\rf+1}
        \frac{U_\ell(u)}{k^{\ell+1/2}} + \rho_{k}^{(r-h)}(u), 
\quad 
    k \ge 2
\end{align}
and
\begin{align} \label{eq:boundukz}
    |\rho_k^{(r-h)} (u)| \le C \frac{(1+u)^{r-h+2}}{k^{(r-h+4)/2}}\log^{\lfloor \frac{r-h+1}{2} \rfloor} k,
\quad 
    k\ge 2.
\end{align}
For $h=r+1$ the sum in \eqref{eq:localtaudecomp} will be empty and \eqref{eq:boundukz} for such $h$ is exactly \eqref{eq:conc.f}. 

Plugging \eqref{eq:localtaudecomp} into \eqref{eq:jumptonegative}, we get
\begin{multline*}
    \sum_{y=0}^\infty
        \left(\frac{y}{\sigma}\right)^{h}
        \pr (S_k = -y, \tau = k)
\\
=
    \sum_{\ell = 0}^{\lf \frac{r-h}{2}\rf}
        \frac{1}{k^{\ell + 3/2}}
        \sum_{y=0}^{\infty}
            \sum_{u=1}^\infty
                \left(\frac{y}{\sigma}\right)^{h}
                U_{\ell+1}(u)
                \pr(X_1=-(y+u))
\\
+
    \sum_{y=0}^{\infty}
        \sum_{u=1}^\infty
            \left(\frac{y}{\sigma}\right)^{h}
            \rho_k^{(r-h)}(u)
            \pr(X_1=-(y+u)).
\end{multline*}
Set 
\begin{align*}
    b_\ell^{(h)}
&:=
    \sum_{y=0}^{\infty}
    \sum_{u=1}^\infty
        \left(\frac{y}{\sigma}\right)^{h}
        U_{\ell+1}(u)\pr(X_1=-(y+u))\\
&=\frac{1}{\sigma^h}\sum_{u=1}^\infty
 U_{\ell+1}(u)\E[(-X_1-u)^h;-X_1>u].
\end{align*}
It follows from \eqref{eq:U-bound} that 
\begin{align*}
    |b_\ell^{(h)}|
 &\le
    C
    \sum_{y=0}^{\infty}
    \sum_{u=1}^\infty
        \left(\frac{y}{\sigma}\right)^{h}
        (1+u)^{2\ell+1}\pr(X_1=-(y+u))
\\
&=
    C
    \sum_{y=0}^\infty
        \left(\frac{y}{\sigma}\right)^{h}
        \e[(1-y-X_1)^{2\ell+1};-X_1>y]
\\
&\le
    C
    \sum_{y=0}^\infty
        \left(\frac{y}{\sigma}\right)^{h}
        \e\left[(1-X_1)^{2\ell+1};-X_1>y\right]
\\
&\le
    \frac{C}{h+1}
    \e[(1-X_1)^{2\ell+h+2};X_1\le0].
\end{align*}
Noting that $2\ell+h+2\le r+2$ and using our moment assumption $\e|X_1|^{r+3}<\infty$, we infer that the numbers $b_\ell^{(h)}$ are finite. Thus, we are left to estimate the remainder 
$$
    v_k^{(h)}
:=
    \sum_{y=0}^{\infty}
    \sum_{u=1}^\infty
        \left(\frac{y}{\sigma}\right)^{h}
        \rho_k^{(r-h)}(u)\pr(X_1=-(y+u)).
$$
Since the decomposition \eqref{eq:jumptonegative} holds only for $k\ge 2$ we should separately consider
\begin{align*}
    v_1^{(h)} 
:=
    \Theta_1^{(h)} - \sum_{\ell=0}^{\lfloor \frac{r-h}{2}\rfloor}
    b_\ell^{(h)}.
\end{align*}
Since all $b_{\ell}^{(h)}$ are bounded and
\begin{align*}
    \Theta_1^{(h)}
=
    \sum_{y=0}^\infty
        \left(\frac{y}{\sigma}\right)^{h}
        \pr(S_1=-y, \tau = 1)
=
    \sum_{y=0}^\infty
        \left(\frac{y}{\sigma}\right)^{h} \pr(X_1 = -y) < \infty,
\end{align*}
we have that $v_1^{(h)}$ is bounded as well. 

If $k \ge 2$ then it follows from \eqref{eq:boundukz} that 
\begin{align*}
    |v_k^{(h)}|
&\le  
    \frac{C\log^{\lfloor \frac{r-h+1}{2} \rfloor} k}
    {k^{(r-h+4)/2}}
    \sum_{y=0}^{\infty}
    \sum_{u=1}^\infty
        \left(\frac{y}{\sigma}\right)^{h}
        (1+u)^{r-h+2}\pr(X_1=-(y+u))\\
&\le
    \frac{C_1\log^{\lfloor \frac{r-h+1}{2} \rfloor} k}
    {k^{(r-h+4)/2}}
    \sum_{y=0}^\infty 
        \left(\frac{y}{\sigma}\right)^{h}
        \e\left[(1-X_1)^{r-h+2};-X_1>y\right]\\
&\le
    \frac{C_2\log^{\lfloor \frac{r-h+1}{2} \rfloor} k}
{k^{(r-h+4)/2}}\E[|X_1|^{r+3}].
\end{align*}
This completes the proof of the first claim.

By the definition of $\Theta_k^{(h)}(z)$,
\begin{equation*}
    \Theta_k^{(h)}-\Theta_k^{(h)}(z)
=
    \sum_{y>\sigma|z|/4}
        \left(\frac{y}{\sigma}\right)^{h}
        \pr(S_k=-y,\tau=k).
\end{equation*}
Applying now \eqref{eq:y-tau} for $k\ge 2$, we have 
\begin{align*}
\Theta_k^{(h)}-\Theta_k^{(h)}(z)
&\le \frac{C}{\sigma^hk^{3/2}}
\sum_{y>|x|/4} y^h\e[-X_1;-X_1>y]\\
&\le \frac{C}{\sigma^hk^{3/2}}\e[(-X_1)^{h+1};-X_1>|x|/4].
\end{align*}
Applying the Markov inequality to the expectation on the right-hand side, we get the desired bound. 
\end{proof}
\begin{remark}
\label{rem:lemma18}
It is clear that \eqref{eq:locprobtau2} allows one to prove, by repeating verbatim the proof of Lemma~\ref{lem:sumovery}, that for every $h\le r$ one has 
$$
    \overline{\Theta}_k^{(h)}(z)
:=
    \sum_{0\le y\le\sigma|z|/4}
        \left(\frac{y}{\sigma}\right)^{h}
        \pr(S_k=-y,\overline{\tau} = k)
=
    \overline{\Theta}_k^{(h)}
+
    O\left(\frac{1}{k^{3/2}(1+|z|)^{r-h+2}}\right)
$$
and 
\begin{align*}
   \overline{\Theta}_{k}^{(h)}
:=
   \sum_{y=0}^\infty
        \left(\frac{y}{\sigma}\right)^{h}
        \pr(S_k=-y, \overline{\tau} = k)
=
    \sum_{\ell=0}^{\lf \frac{r-h}{2}\rf}
        \frac{\overline{b}_\ell^{(h)}}{k^{\ell+3/2}} 
        + \overline{v}^{(h)}_k,
\end{align*}
where
\begin{align*}
|\overline{v}_1^{(h)}| \le C, \quad
    |\overline{v}_k^{(h)}|
\le
    \frac{C\log^{\lfloor \frac{r-h+1}{2} \rfloor} k}{k^{(r-h+4)/2}},
\quad
    k\ge 2.
\end{align*}
The numbers $\overline{b}_\ell^{(h)}$ are given by the equality
$$
\overline{b}_\ell^{(h)}
=\frac{1}{\sigma^h}\sum_{u=1}^\infty
 \overline{U}_{\ell+1}(u)\E[(-X_1-u)^h;-X_1>u].
$$
This is the only change one needs to prove Theorem~\ref{thm:lattice2}.
\hfill$\diamond$
\end{remark}

Using first the second claim in Lemma~\ref{lem:sumovery}
and applying then \eqref{eq:simple-ineq}, we derive from \eqref{eq:Tjqs} 
\begin{multline*}
    T_{j,q,s}(z)
=
    \sum_{k=1}^{n-1}
    e^{-\frac{z^2}{2(n-k)}}
    \sum_{\nu+2\mu\le r-s+1}
            \frac{(-1)^\nu z^{q-s+\nu}h_{2\mu}(0)}{\nu!(2\mu)!
            (n-k)^{j+\nu+\mu+1/2}}
            \Theta_k^{(s+\nu+2\mu)}
\\
+
    O\left(\frac{1}{(\min(\sqrt{n},1+|z|))^{r+3}}\right).
\end{multline*}
Using now the first part of Lemma~\ref{lem:sumovery}, we conclude that for some constants $\bar{c}_{\nu,\mu,s}$ it holds
\begin{multline*}
    T_{j,q,s}(z) 
=
    \sum_{\nu+2\mu+2\ell \le r-s+1}
        \frac{(-1)^\nu h_{2\mu}(0)}{\nu!(2\mu)!}
        b_{\ell}^{(s+\nu+2\mu)}
\sum_{k=1}^{n-1}
    \frac{z^{q-s+\nu}
    e^{-\frac{z^2}{2(n-k)}}
    }{
    k^{\ell+3/2}
    (n-k)^{j+\nu+\mu+1/2}
    }
\\
+
    \sum_{\nu+2\mu\le r-s+1}
        \bar{c}_{\nu,\mu,s}
        z^{q-s+\nu}
            \sum_{k=1}^{n-1}
                \frac{
                v_k^{(s+\nu+2\mu)}
                e^{-\frac{z^2}{2(n-k)}}
                }
                {
                (n-k)^{j+\nu+\mu+1/2}
                }
+
    O\left(\frac{1}{(\min(\sqrt{n},1+|z|))^{r+3}}\right),
\end{multline*}
where
\begin{align} \label{eq:vkestimation}
|v_1^{(s+\nu+2\mu)}| \le C, \quad
    |v_k^{(s+\nu+2\mu)}|
\le 
    \frac{C 
    \log^{
    \lf \frac{r-s-\nu-2\mu+1}{2} \rf
    }
    k
    }{
    k^{(r-s-\nu-2\mu+4)/2}
    }.
\end{align}
Substituting this decomposition of $T_{j,q,s}(z)$ into \eqref{eq:pxydecomp1}, we obtain
\begin{align} \label{eq:Pdecomp}
        \mathcal{P}_n(x)
=
    \Sigma_1(x/\sigma)
+
    \Sigma_2(x/\sigma)
+
    O\left(\frac{1}{(\min(\sqrt{n},1+|x|))^{r+3}}\right),
\end{align}
where 
\begin{multline} \label{eq:sigma1def}
    \Sigma_1(z)
=
    \sum_{j=0}^{2r+2}
    \sum_{q=0}^{\lf \frac{3j}{2}\rf}
    \sum_{s=0}^{q \wedge (r+1)}
    \sum_{\nu+2\mu+2\ell \le r-s+1}
            a_{q,j} 
            \binom{q}{s}
            \frac{(-1)^\nu h_{2\mu}(0)}{\nu!(2\mu)!}
            b_{\ell}^{(s+\nu+2\mu)}
\\                    
\times 
    \sum_{k=1}^{n-1}
        \frac{
        z^{q-s+\nu}
        }{k^{\ell+3/2}
        (n-k)^{j+\nu+\mu+1/2}}
        e^{-\frac{z^2}{2(n-k)}}
\end{multline}
and
\begin{align*}
    \Sigma_2(z)
=
    \sum_{j=0}^{2r+2}
    \sum_{q=0}^{\lf \frac{3j}{2}\rf}
    \sum_{s=0}^{q \wedge (r+1)}
    \sum_{\nu+2\mu\le r-s+1}
        \bar{c}_{\nu,\mu,s}
        a_{q,j} 
            \binom{q}{s}
        z^{q-s+\nu}
        \sum_{k=1}^{n-1}
            \frac
            {
                v_k^{(s+\nu+2\mu)}
                e^{-\frac{z^2}{2(n-k)}}
            }
            {
                (n-k)^{j+\nu+\mu+1/2}
            }.
\end{align*}
We next derive a decomposition for the terms collected in $\Sigma_2(z)$.
\begin{lemma}
\label{lem:restvk} 
For every tuple $(j,q,s,\nu,\mu)$ satisfying $s+\nu+2\mu \le r+1$, $q\le \lf \frac{3j}{2}\rf$, there exist polynomials $F_\eta$ of degree $\eta$, 
$\eta\le\lf \frac{r+3}{2} \rf$ such that 
\begin{multline*}
    z^{q-s+\nu}
    \sum_{k=1}^{n-1}
        \frac{
        v_k^{(s+\nu+2\mu)}
        e^{-\frac{z^2}{2(n-k)}}
        }
        {
        (n-k)^{j+\nu+\mu+1/2}
        }
=
    \frac{z^{q-s+\nu}}{n^{j+\nu+\mu+1/2}}
    e^{-\frac{z^2}{2n}}
    \sum_{\eta=0}^{\lf \frac{r+3}{2} \rf}
        \frac{F_\eta(z^2/n)}{n^{\eta}}
\\
+
    O\left(\frac{\log^{\lf \frac{r+2}{2} \rf}n}{(\min(\sqrt{n},1+|z|))^{r+3}}\right).
\end{multline*}
\end{lemma}
\begin{proof}
To bound the sum over $k\ge n/2$ we apply \eqref{eq:vkestimation} and Lemma~\ref{lem:firsthalfsum} to get
\begin{align*}
    \left|z^{q-s+\nu}
        \sum_{k=\lf n/2\rf +1}^{n-1}
            \frac{
            v_k^{(s+\nu+2\mu)}
            e^{-\frac{z^2}{2(n-k)}}
            }
            {
            (n-k)^{j+\nu+\mu+1/2}
            }
    \right|
\le
    \frac{C\log^{\lf \frac{r-s-\nu-2\mu+1}{2} \rf}n
    |z|^{q-s+\nu}}{n^{(r-s-\nu-2\mu+4)/2}
    |z|^{2j+2\nu+2\mu-1}}.
\end{align*}
The power of $|z|$ on the right-hand side equals $(q-2j)+1-(s+\nu+2\mu)$. This number equals $1$ if and only if $q=2j$ and $s+\nu+2\mu=0$. In all other cases, this number is non-positive. In the latter case, recalling that $q\le 2j$, we obtain
\begin{align}
\label{eq:lem:rem3:0}
\nonumber
\left|
    z^{q-s+\nu}
    \sum_{k=\lf n/2\rf +1}^{n-1}
        \frac{v_k^{(s+\nu+2\mu)}
        e^{-\frac{z^2}{2(n-k)}}}
        {(n-k)^{j+\nu+\mu+1/2}}
\right|
&\le
    \frac{C\log^{\lf \frac{r+1}{2} \rf}n}{n^{(r-s-\nu-2\mu+4)/2}|z|^{s+\nu+2\mu-1}}\\
&=
    O\left(\frac{\log^{\lf \frac{r+1}{2} \rf}n}{(\min(\sqrt{n},1+|z|))^{r+3}}\right).
\end{align}
In the last step we have used inequality $s+\nu+2\mu \le r+1$. 

Assume now that $q=2j$ and $s+\nu+2\mu=0$. The assumption $q\le \lf \frac{3j}{2}\rf$ implies that $j=q=0$. In this case, using \eqref{eq:vkestimation} with $s+\nu+2\mu=0$, we have 
\begin{align*}
    \left|z^{q-s+\nu}
        \sum_{k=\lf n/2\rf +1}^{n-1}
            \frac{
            v_k^{(s+\nu+2\mu)}
            e^{-\frac{z^2}{2(n-k)}}
            }
            {
            (n-k)^{j+\nu+\mu+1/2}
            }
    \right|
=&
    \left|
        \sum_{k=\lf n/2\rf +1}^{n-1}
            \frac{
            v_k^{(0)}
            e^{-\frac{z^2}{2(n-k)}}
            }
            {
            (n-k)^{1/2}
            }
    \right|\\
\le&
    \frac{C}{n^{(r+4)/2}}
    \sum_{k=n/2}^{n-1}
        \frac{1}{\sqrt{n-k}}
\le
    \frac{C}{n^{(r+3)/2}}.
\end{align*}
Thus, \eqref{eq:lem:rem3:0} holds for all values of $q,j,s,\nu$ and $\mu$.

To study the sum over $k\le n/2$ we apply Lemma~\ref{lem:secondhalfsum} with 
\begin{align*}
    a=\frac{r-s-\nu-2\mu+4}{2} \quad \text{and} \quad  b=j+\nu+\mu+1/2.
\end{align*}
As a result we have
\begin{align*}
&
    z^{q-s+\nu}
    \sum_{k=1}^{\lf n/2\rf}
        \frac{v_k^{(s+\nu+2\mu)}
        e^{-\frac{z^2}{2(n-k)}}}
        {(n-k)^{j+\nu+\mu+1/2}}
\\
&\hspace{1cm}
=
    \frac{z^{q-s+\nu}}{n^{j+\nu+\mu+1/2}}
    e^{-\frac{z^2}{2 n}}
    \sum_{\eta=0}^{\lf a \rf-1}
        \frac{F_\eta(z^2/n)}{n^{\eta}}
+
    O\left(\frac{\log^{{
    \lf \frac{r-s-\nu-2\mu+1}{2} \rf
    }+\delta_a} n}{n^{(r+2j-q+3)/2}}\right).
\end{align*}
If $\lf a\rf-1$ is strictly smaller than $\lf \frac{r+3}{2} \rf$ then we can still write the sum up to $\lf \frac{r+3}{2} \rf$. Then, due to $q\le2j$, we have
\begin{multline}\label{eq:lem:rem3:1}
    z^{q-s+\nu}
    \sum_{k=1}^{\lf n/2\rf}
        \frac{
        v_k^{(s+\nu+2\mu)}
        e^{-\frac{z^2}{2(n-k)}}
        }
        {
        (n-k)^{j+\nu+\mu+1/2}
        }
\\
=
    \frac{z^{q-s+\nu}}
    {n^{j+\nu+\mu+1/2}}
    e^{-\frac{z^2}{2 n}}
    \sum_{\eta=0}^{\lf \frac{r+3}{2} \rf}
        \frac{F_\eta(z^2/n)}{n^{\eta}}
+
    O\left(\frac{\log^{{
    \lf \frac{r+2}{2} \rf
    }} n}{n^{(r+3)/2}}\right).
\end{multline}
The power of the logarithm is obtained from the equality $\lf \frac{r+2}{2} \rf = \lf \frac{r+1}{2} \rf + {\bf{1}}_{\{r/2 \in \mathbb{Z} \}}$ and is attained at $a = \frac{r+4}{2}$.
Combining \eqref{eq:lem:rem3:0} and \eqref{eq:lem:rem3:1} we obtain the desired expansion.
\end{proof}
Applying Lemma \ref{lem:restvk} to every summand 
in $\Sigma_2(z)$, we get
\begin{multline*}
    \Sigma_2(z)
=    
    e^{-\frac{z^2}{2n}}
    \sum_{j=0}^{2r+2}
    \sum_{q=0}^{\lf \frac{3j}{2}\rf}
    \sum_{s=0}^{q \wedge (r+1)}
    \sum_{\nu+2\mu\le r-s+1}
        \bar{c}_{\nu,\mu,s}
        a_{q,j} 
            \binom{q}{s}
    \sum_{\eta=0}^{\lf \frac{r+3}{2}\rf}
        \frac{z^{q-s+\nu}F_\eta(z^2/n)}{n^{j+\nu+\mu+\eta+1/2}}
\\
+
    O\left(\frac{\log^{\lf \frac{r+2}{2} \rf}n}{(\min(\sqrt{n},1+|z|))^{r+3}}\right).
\end{multline*}
Noting that all monomials in this sum are of the form
$\frac{z^\alpha}{n^{\beta+1/2}}$ with $\alpha \le 2\beta$, and using \eqref{eq:simple-ineq},
we can rewrite the previous expression for $\Sigma_2(z)$ as follows
\begin{align} \label{eq:sigma2}
    \Sigma_2(z)
=
    e^{-\frac{z^2}{2n}}
    \sum_{\nu=1}^{r+2}
        \frac{P_\nu^{(3)}\big(z/\sqrt{n}\big)}{n^{\nu/2}}
+
    O\left(\frac{\log^{\lf \frac{r+2}{2} \rf}n}{(\min(\sqrt{n},1+|z|))^{r+3}}\right),
\end{align}
where $P_\nu^{(3)}$ are some polynomials.

Recall the formula for $\Sigma_1(z)$ and substitute the exact expression \eqref{eq:mucoefficients} for $h_{2\mu}(0)$:
\begin{multline*}
    \Sigma_1(z)
=
    \sum_{j=0}^{2r+2}
    \sum_{q=0}^{\lf \frac{3j}{2}\rf}
    \sum_{s=0}^{q \wedge (r+1)}
    \sum_{\nu+2\mu+2\ell \le r-s+1}
            a_{q,j} 
            \binom{q}{s}
            \frac{(-1)^{\nu+\mu}}{2^\mu\nu!\mu!}
            b_{\ell}^{(s+\nu+2\mu)}
\\                    
\times 
    \sum_{k=1}^{n-1}
        \frac{
        z^{q-s+\nu}
        }{k^{\ell+3/2}
        (n-k)^{j+\nu+\mu+1/2}}
        e^{-\frac{z^2}{2(n-k)}}.
\end{multline*}
Applying Proposition~\ref{prop:basisconvolution} to each sum $\sum_{k=1}^{n-1}$ and using additionally Lemma~\ref{lem:qjellsmall} to all terms that contain $Q_{a,b,c}$ with $2(a+b)-c\ge r+3$, we obtain
\begin{multline}\label{eq:sigma1decomp}
    \Sigma_1(z)
=
    {\rm sgn}(z) e^{-\frac{z^2}{2n}}
    \sum_{\eta=2}^{r+2}
        \frac{Q_{\eta}(z/\sqrt{n})}{n^{\eta/2}}
+
    e^{-\frac{z^2}{2n}}
    \sum_{\nu=1}^{r+2}
    \frac{P_{\nu}^{(4)}(z/\sqrt{n})}{n^{\nu/2}}
\\
+
    O\left(\frac{1}{(\min(\sqrt{n},1+|z|))^{r+3}}\right),
\end{multline}
where $P_{\nu}^{(4)}$ are some polynomials. It is clear from Proposition~\ref{prop:basisconvolution} that $Q_{a,b,c}$ is a summand of $Q_\eta$ if and only if $2(a+b)-c=\eta$. Therefore, for $\eta = 2,\dots,r+2$ one has
\begin{align}\label{eq:qnu}
    Q_\eta
= 
    \sqrt{2\pi}
    \sum_
    {
    \begin{smallmatrix}
    j,q,s,\nu,\mu,\ell \in \mathbb{Z}_{\ge0}\\
    s\le q\le \lf \frac{3j}{2} \rf\\
    \eta-2=2(j+\mu+\ell)+\nu-q + s
    \end{smallmatrix}
    }
        a_{q,j} \binom{q}{s}\frac{(-1)^\nu h_{2\mu}(0)}{\nu!(2\mu)!}
        b_{\ell}^{(s+\nu+2\mu)}
        Q_{
        \ell+1, j+\nu+\mu, q-s+\nu
        }.
\end{align}
Substituting \eqref{eq:sigma1decomp} and \eqref{eq:sigma2} into \eqref{eq:Pdecomp} and using \eqref{eq:maindecomp0.5} we finally obtain
\begin{multline} \label{eq:maindecomp5}
    \pr(S_n = x, \tau < n)
=
    {\rm sgn}(z) e^{-\frac{z^2}{2n}}
    \sum_{\nu=2}^{r+2}
        \frac{Q_{\nu}(z/\sqrt{n})}{n^{\nu/2}}
+
    e^{-\frac{z^2}{2n}}
    \sum_{\nu=1}^{r+2}
    \frac{P_{\nu}^{(5)}(z/\sqrt{n})}{n^{\nu/2}}
\\
+
o\left(\frac{1}{(\sqrt{n}+|x|)^{r+2}}
\right)
+
O\left(
    \frac{\log^{{
    \lf \frac{r+2}{2} \rf
    }} n}{(\min(\sqrt{n},1+|x|))^{r+3}}\right),
\end{multline}
where as usual $z = x/\sigma$.

Combining this expansion for $\pr(S_n = x, \tau < n)$
with the known expansions for $\pr(S_n = x)$ we get an expansion for $\pr(S_n = x, \tau > n)$, but the result will differ from the claim in Theorem~\ref{thm:lattice}. To finish the proof of Theorem~\ref{thm:lattice} we need to show that the functions $Q_\nu$ are polynomials. This will be done in the next section. Furthermore, we have to determine the degrees of polynomials $P_\nu$. This final part of the proof is performed in Lemma~\ref{lem:degree}.
\section{Simplification of the expansion}
In the previous section we have derived an expansion
for $\pr(S_n=x,\tau<n)$. One can plug this into \eqref{eq:reflection}
and obtain, after expanding $\pr(S_n=x)$, an expansion for $\pr(S_n=x,\tau>n)$. But this expansion would differ formally from the claim in Theorem~\ref{thm:lattice}. More precisely, at the moment we have only shown that the functions $Q_\nu$ are polynomials in $z/\sqrt{n}$ and 
$\sqrt{n}/z$, but in Theorem~\ref{thm:lattice} we have only polynomials in $z/\sqrt{n}$. A further weakness of \eqref{eq:maindecomp5} is the fact that we do not have explicit formulas for the polynomials $P_\nu^{(5)}$; only their existence has been established. To obtain a more practical expression we shall use the fact that \eqref{eq:maindecomp5} is valid for {\it all} $x$.

If $x$ is non-positive then 
$$
\pr(S_n=x)=\pr(S_n=x,\tau\le n)=
\pr(S_n=x,\tau<n)+\pr(S_n=x,\tau=n).
$$
Equivalently,
\begin{equation}
\label{eq:x_negativ}
\pr(S_n=x,\tau<n)=\pr(S_n=x)-\pr(S_n=x,\tau=n),
\quad x\le0.
\end{equation}
This equality allows us to compare asymptotic expansions for 
$\pr(S_n=x)$ and $\pr(S_n=x,\tau<n)$ and to gain some additional information on $Q_\nu$ and $P_\nu^{(5)}$.

To compare the expansions for $\pr(S_n=x,\tau<n)$ and 
$\pr(S_n=x)$, we have to estimate $\pr(S_n=x,\tau=n)$.
We know from \eqref{eq:y-tau} that 
$$
\pr(S_n=x,\tau=n)\le C\frac{\e[-X_1;-X_1>|x|]}{n^{3/2}}.
$$
Due to the assumption $\e|X_1|^{r+3}<\infty$,
$$
\e[-X_1;-X_1>|x|]\le C\frac{1}{1+|x|^{r+2}}.
$$
Thus, 
$$
\pr(S_n=x,\tau=n)\le C\frac{1}{n^{3/2}(1+|x|^{r+2})}.
$$
Plugging this into \eqref{eq:x_negativ}, we have 
\begin{equation}
\label{eq:x_negativ2}
\pr(S_n=x,\tau<n)=\pr(S_n=x)
+O\left(\frac{1}{n^{3/2}(1+|x|^{r+2})}\right).
\end{equation}
We know from \eqref{eq:asymp4} that there exist polynomials $\widehat{q}_\nu$ such that with $z=x/\sigma$ it holds
\begin{align}
\label{eq:uncond}
    \pr(S_n=x)
 =
    e^{-\frac{z^2}{2n}}
    \sum_{\nu=1}^{r+2}
        \frac{\widehat{q}_\nu(z/\sqrt{n})}{n^{\nu/2}}
+
    o\left(\frac{1}{n^{(r+2)/2}}\right)
\end{align}
uniformly in $x$.
Plugging this into \eqref{eq:x_negativ2}, we have 
\begin{align*}
    \pr(S_n=x,\tau<n)
=
    e^{-\frac{z^2}{2n}}
    \sum_{\nu=1}^{r+2}
        \frac{\widehat{q}_\nu(z/\sqrt{n})}{n^{\nu/2}}
+
    o\left(\frac{1}{n^{(r+2)/2}}\right)
+
    O\left(\frac{1}{n^{3/2}(1+|x|^{r+2})}\right)
\end{align*}
for all $x\le0$. Comparing this with \eqref{eq:maindecomp5}, we conclude that
\begin{align*}
    e^{-\frac{z^2}{2n}}
    \sum_{\nu=1}^{r+2}
        \frac{\widehat{q}_\nu(z/\sqrt{n})}{n^{\nu/2}}
=
    -e^{-\frac{z^2}{2n}}
    \sum_{\nu=2}^{r+2}
        \frac{Q_{\nu}(z/\sqrt{n})}{n^{\nu/2}}
+
    e^{-\frac{z^2}{2n}}
    \sum_{\nu=1}^{r+2}
        \frac{P_{\nu}^{(5)}(z/\sqrt{n})}{n^{\nu/2}}\\
+
o\left(\frac{1}{n^{(r+2)/2}}
\right)+
O\left(\frac{\log^{{
    \lf \frac{r+2}{2} \rf
    }} n}{(\min(\sqrt{n},1+|x|))^{r+3}}\right)
\end{align*}
for all $x\le0$.
Setting $x=u\sqrt{n}/\sigma$ for some fixed $u<0$, we get 
\begin{align*}
    \sum_{\nu=1}^{r+2}
        \frac{\widehat{q}_\nu(u)}{n^{\nu/2}}
+
    o\left(\frac{1}{n^{(r+2)/2}}\right)
=
-
    \sum_{\nu=2}^{r+2}
    \frac{Q_{\nu}(u)}{n^{\nu/2}}
+
    \sum_{\nu=1}^{r+2}
        \frac{P_{\nu}^{(5)}(u)}{n^{\nu/2}}.
\end{align*}
Multiplying both sides by $\sqrt{n}$ and letting $n\to\infty$, we get 
$$
\widehat{q}_1(u)=P_1^{(5)}(u),\quad u<0.
$$
After cancellation of those terms we can multiply both sides by $n$ to obtain 
$$
\widehat{q}_2(u)=-Q_2(u)+P_2^{(5)}(u),\quad u<0.
$$
Repeating this procedure $r+2$ times, we have 
$$
\widehat{q}_\nu(u)=-Q_\nu(u)+P_\nu^{(5)}(u),\quad u<0
$$
for all $\nu\in\{2,3,\ldots,r+2\}$. Since $\widehat{q}_\nu$ and 
$P_\nu^{(5)}$ are polynomials, we infer that $Q_\nu$ are polynomials as well. Furthermore, equality of polynomials on the negative half-axis implies equality on the whole real line.
Consequently,
\begin{align}
\label{eq:nu=1}
\widehat{q}_1(u)=P_1^{(5)}(u)
\end{align}
and 
\begin{align}
\label{eq:nu>1}
\widehat{q}_\nu(u)=-Q_\nu(u)+P_\nu^{(5)}(u),\quad 
\nu\in\{2,3,\ldots,r+2\}.
\end{align}
Applying \eqref{eq:uncond} and \eqref{eq:maindecomp5} to the corresponding terms on the right-hand side of \eqref{eq:reflection},
we obtain 
\begin{align*}
&
    \pr(S_n=x,\tau>n)
\\
&\hspace{1cm}
=
    e^{-\frac{z^2}{2n}}
    \left( 
        \sum_{\nu=1}^{r+2}
            \frac{\widehat{q}_\nu(z/\sqrt{n})}{n^{\nu/2}}
    -
        \sum_{\nu=2}^{r+2}
            \frac{Q_\nu(z/\sqrt{n})}{n^{\nu/2}}
    -
        \sum_{\nu=1}^{r+2}
            \frac{P^{(5)}_\nu(z/\sqrt{n})}{n^{\nu/2}}
    \right)
\\
&\hspace{3cm}
+
    o\left(\frac{1}{(\sqrt{n}+|x|)^{r+2}}
\right)
+
O\left(
    \frac{\log^{{
    \lf \frac{r+2}{2} \rf
    }} n}{(\min(\sqrt{n},1+|x|))^{r+3}}\right).
\end{align*}
Combining this with the equalities \eqref{eq:nu=1} and \eqref{eq:nu>1}, we finally get, with $t=\frac{x}{\sigma \sqrt{n}}$
\begin{align*} 
&\pr(S_n=x,\tau>n) \nonumber
=
    -2e^{-\frac{t^2}{2}}
    \sum_{\eta=2}^{r+2}
        \frac{Q_\eta(t)}{n^{\eta/2}}
\\
&\hspace{3cm}
+
    o\left(\frac{1}{(\sqrt{n}+|x|)^{r+2}}
\right)
+
O\left(
    \frac{\log^{{
    \lf \frac{r+2}{2} \rf
    }} n}{(\min(\sqrt{n},1+|x|))^{r+3}}\right).
\end{align*}
This proves Theorem~\ref{thm:lattice} with $P_\nu(t)=-2Q_\nu(t)$,
$\nu=2,3,\ldots,r+2.$
\section{Explicit expressions for $Q_2$ and $Q_3$: proof of Corollary~\ref{cor:r=2}}
In Section~\ref{sec:4} we have shown that the coefficients $Q_\nu$ of our expansion are given by:
\begin{align} \label{eq:qgeneral}
    Q_\eta
= 
    \sqrt{2\pi}
    \sum_
    {
    \begin{smallmatrix}
    j,q,s,\nu,\mu,\ell \in \mathbb{Z}_{\ge0}\\
    s\le q\le \lf \frac{3j}{2} \rf\\
    \eta-2=2(j+\mu+\ell)+\nu-q+s
    \end{smallmatrix}
    }
        a_{q,j} \binom{q}{s}\frac{(-1)^{\nu+\mu}}{2^\mu\nu!\mu!}
        b_{\ell}^{(s+\nu+2\mu)}
        Q_{\ell+1, j+\nu+\mu,q-s+\nu}.
\end{align}
All ingredients in the sum are known:
\begin{itemize}
\item $a_{q,j}$ are the coefficients of the polynomials $P_j^{(0)}$
      which come from the asymptotic expansions for the unconditioned probability $\pr(S_n=x)$, see Proposition~\ref{prop:locprob},
\item the numbers $b_{\ell}^{(s+\nu+2\mu)}$ are defined in    
      Lemma~\ref{lem:sumovery},  
\item for the functions $Q_{j,\ell,m}$ we have the following representation (see Proposition~\ref{prop:basisconvolution}):
\begin{align}\label{eq:qjellm} 
    Q_{j,\ell,m}(t)
=
    t^m
    \sum_{q=0}^{j}
        \gamma_{q,j,\ell} t^{2q}
        \sum_{k=0}^{\ell+j+q-1}
            \frac{(2k-1)!!
            \binom{\ell+j+q-1}{k}}
            {t^{2k+1}},
\end{align}      
\item reals $\gamma_{q,j,\ell}$ in the definition of $Q_{j,\ell,m}$ are given by (see Lemma~\ref{lem:gammas}):
\begin{align}
\label{eq:gammas1}
    \gamma_{q,j,\ell}
=
    \frac{(-1)^q2^j}{2^q(2j-1)!!}
    \sum
        \prod_{i=1}^{j-q}
            (\ell+2a_i-i-1/2)
\end{align}
where the sum is taken over all strictly increasing tuples $(a_1, a_2, \dots, a_{j-q})$ chosen from $(1,2,\dots, j)$. If $j=q$ then we set the sum equal to $1$.
\end{itemize}
Thus, we are able to compute all functions $Q_\nu$. It is also clear that the amount of calculations increases when $\eta$ grows. In this section we make all the calculations in the cases $\eta=2$ and $\eta=3$.
This will give the main term and one correction term in the expansion for
$\pr(S_n=x,\tau>n)$. In the course of our calculations we shall also see that all negative powers in $Q_2$ and $Q_3$ disappear. We already know that this happens, but we were using some indirect arguments. Thus, the calculations below can be seen as an additional explanation
of that effect.

Before we switch to particular cases $\eta=2$ and $\eta=3$, we determine the degree of every $Q_\eta$.
\begin{lemma}
\label{lem:degree} One has
\begin{align*}
    \deg Q_\eta = 3\eta - 5.
\end{align*}
\end{lemma}
\begin{proof}
    By definition $\deg Q_{j,\ell,m} = m+2j-1$. So according to \eqref{eq:qjellm}
\begin{align*}
    \deg Q_\eta 
\le 
    \max\limits_{\eta-2=2(j+\mu+\ell)+\nu-q+s} 
    \big(
        q-s+\nu+2\ell+1
        \big).
\end{align*}
It is easy to see that if $\mu$ or $s$ are positive we can get a bigger value by letting them be zero and adding $2\mu+s$ to $\nu$. Denoting $p:=\nu+2\ell$, one has
\begin{align*}
    \deg Q_\eta 
\le 
    \max\limits_{\eta-2=p + 2j -q}
    (q+p+1).
\end{align*}
Since $q\le \lf \frac{3j}{2}\rf \le \frac{3j}{2}$ we have $2j-q \ge q/3$ hence $p + q/3 \le \eta-2$. So
\begin{align*}
    \deg Q_\eta 
\le 
    \max\limits_{\eta-2=p + 2j -q}
    (q+p+1)
\le
    \max\limits_{p +q/3 \le \eta-2}
    (q+p+1).
\end{align*}
Clearly, the maximum is uniquely attained at $q=3(\eta-2), p = 0$. So
\begin{align*}
    \deg Q_\eta 
\le
    3(\eta-2) + 1 
=
    3\eta  - 5.
\end{align*}
A maximizer is given by $j=2(\eta-2),q=3(\eta-2), \mu=\nu=s=\ell=0$ and by construction it is a unique maximum; hence the degree of $Q_\eta$ is exactly $3\eta-5$.
\end{proof}

To give exact expressions for $Q_2$ and $Q_3$ we need first to find all tuples $(\ell,j,s,q,\nu,\mu)$ such that 
$s\le q\le\lf\frac{3j}{2}\rf$ and 
\begin{align*}
    \eta-2=2(j+\mu+\ell)+\nu-q+s.
\end{align*}
Note that for $\eta=2,3$ it holds $\eta-2 \in \{0,1\}$. Combining that with the restriction $q \le \lf \frac{3j}{2}\rf$, we infer that 
$\mu = \ell = 0$. If $s\ge 1$ then $q\ge1$ and hence $2j-q\ge 1$ but then $2j-q+s \ge 2$. Hence $s=0$ as well. Consequently, we are left to find all integer solutions to the equation
\begin{align*}
    \eta - 2 = 2j - q + \nu,
    \quad 
    j,q,\nu \in \mathbb{Z}_{\ge 0}, \; \; \; q\le \lf \frac{3j}{2}\rf .
\end{align*}
If $\eta = 2$ then the only solution is $\nu=j=q=0$. In the case $\eta=3$ we have three solutions:
$$(\nu, j, q) \in \{(0,2,3),(0,1,1),(1,0,0)\}.$$
This implies that 
\begin{align}
\label{eq:Q2Q3}
\nonumber
    Q_2 
&=
    \sqrt{2\pi}a_{0,0} b_{0}^{(0)} Q_{1,0,0},
\\
    Q_3
&=
    \sqrt{2\pi}a_{3,2} b_{0}^{(0)} Q_{1,2,3} 
+
    \sqrt{2\pi}
    \big(
        a_{0,0}b_{0}^{(1)} 
    +
        a_{1,1}b_0^{(0)}
    \big) Q_{1,1,1}.
\end{align}
In the proof of Proposition~\ref{prop:locprob} we have defined 
$P_{0}^{(0)}(z) = 1/(\sigma\sqrt{2\pi})$. This implies that 
$a_{0,0} = 1/(\sigma\sqrt{2\pi})$. To find the values of $a_{1,1}$
and $a_{3,2}$ we first notice that, due to \eqref{eq:qpolydef}, 
\begin{align*}
   \widehat{q}_1(t) 
=
   \frac{m_3}{{6}\sqrt{2\pi}\sigma^3}H_3(t)
=
    \frac{m_3}{{6}\sqrt{2\pi}\sigma^3}(t^3-3t),
\end{align*}
where $m_3$ is the third central moment and $H_3$ is the third Hermite polynomial. Combining this with \eqref{eq:qpolycoeff}, we get 
$$
A_{1,0}=\frac{m_3}{6\sigma^3\sqrt{2\pi}}
\quad\text{and}\quad 
A_{1,1}=-\frac{m_3}{2\sigma^3\sqrt{2\pi}}.
$$
Noting that \eqref{eq:aqjviaA} implies that $a_{1,1}=A_{1,1}$ and 
$a_{3,2}=A_{1,0}$, we finally have
\begin{align}
\label{eq:a-coef}
a_{0,0}=\frac{1}{\sigma\sqrt{2\pi}},\quad
a_{1,1} =-\frac{m_3}{2\sqrt{2\pi}\sigma^3}
\quad\text{and}\quad
a_{3,2} =\frac{m_3}{6\sqrt{2\pi}\sigma^3}.
\end{align}

Substituting values of $a_{q,j}$ from \eqref{eq:a-coef} into \eqref{eq:Q2Q3} and writing, for brevity, $\theta_0$ and $\theta_1$
instead of $b_0^{(0)}$ and $b_0^{(1)}$ we have
\begin{align}
\label{eq:Q2Q3-1}
\nonumber
    &Q_2(t)
=
    \frac{\theta_0}{\sigma}Q_{1,0,0}(t),\\
    &Q_3(t)
=
    \frac{\theta_0m_3}{6\sigma^3}Q_{1,2,3}(t)
+
    \left(\frac{\theta_1}{\sigma}-\frac{\theta_0m_3}{2\sigma^3}\right)
    Q_{1,1,1}(t).
\end{align}
Thus, it remains to determine the functions $Q_{j,\ell,m}$ appearing in \eqref{eq:Q2Q3-1}. In view of \eqref{eq:qjellm} we only need to determine the corresponding numbers $\gamma_{q,j,\ell}$, for which we have \eqref{eq:gammas1}. Noting that $\gamma_{0,1,0} = 1$ and 
$\gamma_{1,1,0} = -1$, we get
\begin{align*}
    Q_{1,0,0} (t)
=
    \gamma_{0,1,0}\frac{1}{t}
+
    \gamma_{1,1,0}
    t^2
    \left(
        \frac{1}{t}
    +
        \frac{1}{t^3}
    \right)
=
    -t.
\end{align*}
Noting next that $\gamma_{0,1,1} = 3, \gamma_{1,1,1} = -1$ we have
\begin{align*}
    Q_{1,1,1}(t) 
&=
    t\gamma_{0,1,1}
    \left(
        \frac{1}{t}
    +
        \frac{1}{t^3}
    \right)
+
    t\gamma_{1,1,1}t^2
    \left(
        \frac{1}{t}
    +
        \frac{2}{t^3}
    +
        \frac{3}{t^5}
    \right)
\\
&=
    3+\frac{3}{t^2}-t^2-2-\frac{3}{t^2}
=
    1-t^2.
\end{align*}
Finally, the equalities $\gamma_{0,1,2} = 5, \gamma_{1,1,2} = -1$
imply that 
\begin{align*}
    Q_{1,2,3}(t) 
&=
    t^3\gamma_{0,1,2}
    \left(
        \frac{1}{t}
    +
        \frac{2}{t^3}
    +
        \frac{3}{t^5}
    \right)
+
    t^3\gamma_{1,1,2}t^2
    \left(
        \frac{1}{t}
    +
        \frac{3}{t^3}
    +
        \frac{9}{t^5}
    +
        \frac{15}{t^7}
    \right)
\\
&=
    5t^2+10+\frac{15}{t^2}-t^4-3t^2-9-\frac{15}{t^2} = -t^4+2t^2+1.
\end{align*}
Here we see that all negative powers of $t$ indeed disappear. This can be seen as a confirmation of our general argument in the previous section.

Substituting these representations for $Q_{1,0,0}$, $Q_{1,1,1}$ and 
$Q_{1,2,3}$ into \eqref{eq:Q2Q3-1}, we finally obtain 
$$
Q_2(t)=-\frac{\theta_0}{\sigma}t
$$
and
\begin{align*}
    Q_3(t)
&=
    \frac{\theta_0m_3}{6\sigma^3}(1+2t^2-t^4)
+
    \left(\frac{\theta_1}{\sigma}-
    \frac{\theta_0m_3}{2\sigma^3}\right)
    (1-t^2)\\
&=
    \frac{\theta_0 m_3}{6\sigma^3}\left(-2+5t^2-t^4\right)
+
    \frac{\theta_1}{\sigma}(1-t^2). 
\end{align*}
Recalling that $P_\nu(t)=-2Q_\nu(t)$ for all $\nu$ completes the proof
of the corollary.

We conclude this section by discussing alternative representations for the numbers $\theta_0$ and $\theta_1$, see \eqref{eq:theta0} and \eqref{eq:theta1}. By the total probability law we have 
$$
\pr(\tau=k)=\sum_{u=1}^\infty\pr(S_{k-1}=u,\tau>k-1)\pr(X_1<-u).
$$
According to Theorem 5 in \cite{VW09},
$$
\lim_{k\to\infty}k^{3/2}\pr(S_{k-1}=u,\tau>k-1)=U_1(u)
$$
for every fixed $u$. Moreover, \eqref{eq:conc.f} implies that 
$$
k^{3/2}\pr(S_{k-1}=u,\tau>k-1)\pr(X_1<-u)\le Cu\pr(X_1<-u).
$$
Since the right-hand side is summable, we may apply the dominated convergence theorem. As a result we have
$$
\lim_{k\to\infty} k^{3/2}\pr(\tau=k)
=\sum_{u=1}^\infty U_1(u)\pr(X_1<-u)=b_0^{(0)},
$$
where the last equality follows from the definition of $b_0^{(0)}$ in the proof of Lemma~\ref{lem:sumovery}. Similar arguments show the validity
of \eqref{eq:theta1}.

\section{Comparison of decompositions}\label{sec:7}
Assume that $\E|X_1|^{2m+1} < \infty$ for some $m\ge 1$.
Setting $r=2m-2$ in \eqref{eq:locprobtau},
we have
\begin{align*}
    \pr(S_n = x, \tau > n) 
=
    \sum_{j=1}^m 
        \frac{U_j(x)}{(n+1)^{j+1/2}} 
+
    O\left(
    \frac{(1+x)^{2m}}{n^{m+1}}
    \log^{m} n
    \right)
\end{align*}
for some functions $U_j$. Changing the basis of the decomposition from $(n+1)^{-j-1/2}$ to $n^{-j-1/2}$, one gets
\begin{align*}
    \pr(S_n = x, \tau > n) 
=
    \sum_{j=1}^m W_j(x)n^{-j-1/2} 
+
    O\left(
    \frac{(1+x)^{2m}}{n^{m+1}}
    \log^{m} n
    \right)
\end{align*}
where every $W_j$ is given by a linear combination of $U_1, U_2, \dots, U_j$. Due to Theorem 1.8 in \cite{DTW23}, functions $U_j$ are almost polynomials and hence the same holds for the functions $W_j$. More precisely, if the random variable $X_1$ possesses finite exponential moments then there exist polynomials $Q_j(x)$, each of degree $2j-1$, such that, for some $\varepsilon > 0$,
\begin{align*}
    W_j(x) = Q_j(x) + O(e^{-\varepsilon x}).
\end{align*}
Let $x = \sigma^{-1} n^{1/2-\alpha}$ for some $\alpha \in (1/4, 3/8)$, then we obtain
\begin{align}
    \pr(S_n = x, \tau > n) 
&=
    \sum_{j=1}^m 
        Q_j(n^{1/2-\alpha})n^{-j-1/2} 
+
    O\left(n^{-m/4}
    \right)\\
&= \nonumber
    \sum_{j=1}^m
    \sum_{k=0}^{2j-1}
        \sigma^{-k}
        [t^k]Q_j(t) n^{k/2-j-1/2 - \alpha k} 
+
    O\left(n^{-m/4}
    \right).
\end{align}
On the other hand, for the same $r$ and $x$ we have from Theorem \ref{thm:lattice}
\begin{align}
    \pr(S_n = x, \tau > n) 
&=
    e^{-\frac{1}{2}n^{-2\alpha}}
    \sum_{\nu=2}^{2m}
        P_\nu(n^{-\alpha}) n^{-\nu/2}
+
    O(n^{-m/4})\\
&= \nonumber
    \left(
    \sum_{\mu = 0}^{m}
        \frac{(-1)^\mu}{2^\mu \mu!}
        n^{-2\mu \alpha}
    \right)
    \left(
    \sum_{\nu=2}^{2m}
        P_\nu(n^{-\alpha}) n^{-\nu/2}
    \right)
+
    O(n^{-m/4}).
\end{align}
Fixing any irrational $\alpha \in (1/4, 3/8)$ and assuming that $m$ is large enough, we can equate the coefficients of identical powers of $n$. This leads to the equality
\begin{align*}
    [t^k] Q_j(t)
=
    \sigma^k
    \sum_{2\mu + q = k}
        \frac{(-1)^\mu}{2^\mu \mu!}
        [t^q] P_{2j+1-k}(t),
\end{align*}
where the sum is taken over the nonnegative integer $q$ and $\mu$. Note that although polyharmonic functions $W_j(x)$ depend heavily on the whole distribution of increments of our random walk, the polynomials $P_\nu$ depend only on some moments and cumulants of that distribution. Moreover, those calculations show that the polynomial part of $W_j$, and hence of $U_j$, is determined by the polynomials $P_\nu$ with $\nu \le 2j+1$.

 
\end{document}